\documentclass[a4paper,12pt,reqno,dvipsnames]{amsart}

\pdfoutput=1

\usepackage[a4paper,top=3cm,bottom=2cm,left=3cm,right=3cm,marginparwidth=1.75cm]{geometry}

\usepackage[all]{xy}
\usepackage{amssymb,amsmath,amsthm,amscd,epsfig}
\usepackage{graphicx}
\usepackage[colorinlistoftodos]{todonotes}
\PassOptionsToPackage{hyphens}{url}
\usepackage[colorlinks=true, allcolors=blue]{hyperref}

\usepackage[english]{babel}
\usepackage[utf8x]{inputenc}
\usepackage[T1]{fontenc}
\usepackage{cleveref}
\usepackage{amsmath, amssymb, amsthm}
\usepackage{mathrsfs}
\usepackage{mathtools}
\usepackage{tikz-cd}

\usepackage[OT2,T1]{fontenc}
\DeclareSymbolFont{cyrletters}{OT2}{wncyr}{m}{n}
\DeclareMathSymbol{\Sha}{\mathalpha}{cyrletters}{"58}
\DeclareMathSymbol{\Be}{\mathalpha}{cyrletters}{"42}

\newtheorem{thm}{Theorem}[section]

\newtheorem{lemma}[thm]{Lemma}
\newtheorem{theorem}[thm]{Theorem}
\newtheorem{proposition}[thm]{Proposition}

\theoremstyle{definition}
\newtheorem{question}[thm]{Question}

\newtheorem{example}[thm]{Example}
\newtheorem{definition}[thm]{Definition}

\newtheorem{remark}[thm]{Remark}

\numberwithin{equation}{section}

\def\Z{\ifmmode{{\mathbb Z}}\else{${\mathbb Z}$}\fi}
\def\Q{\ifmmode{{\mathbb Q}}\else{${\mathbb Q}$}\fi}
\def\C{\ifmmode{{\mathbb C}}\else{${\mathbb C}$}\fi}
\def\P{\ifmmode{{\mathbb P}}\else{${\mathbb P}$}\fi}
\def\rmH{\ifmmode{{\mathrm H}}\else{${\mathrm H}$}\fi}
\def\G{\ifmmode{{\mathbb G}}\else{${\mathbb G}$}\fi}
\def\R{\ifmmode{{\mathbb R}}\else{${\mathbb R}$}\fi}
\def\F{\ifmmode{{\mathbb F}}\else{${\mathbb F}$}\fi}
\def\N{\ifmmode{{\mathbb N}}\else{${\mathbb N}$}\fi}
\def\O{\ifmmode{{\calO}}\else{${\calO}$}\fi}
\def\D{\ifmmode{{\cal{D}}^b}\else{${{\cal{D}}^b}$}\fi}
\renewcommand\O{\mathcal{O}}

\newcommand\GG{\mathbb{G}}

\newcommand\Gm{\GG_\mathrm{m}}

\DeclareMathOperator{\HH}{H}

\DeclareMathOperator{\lcm}{lcm}

\DeclareMathOperator{\Frob}{Frob}

\DeclareMathOperator{\inv}{inv}

\DeclareMathOperator{\Hom}{Hom}

\DeclareMathOperator{\Gal}{Gal}

\DeclareMathOperator{\ord}{ord}

\DeclareMathOperator{\Spec}{Spec}

\DeclareMathOperator{\tors}{tors}

\DeclareMathOperator{\res}{\partial}

\DeclareMathOperator{\rank}{rank}

\DeclareMathOperator{\GL}{GL}

\DeclareMathOperator{\sign}{sign}

\usepackage{todonotes}

\newcommand{\A}{\mathbb{A}}

\DeclareMathOperator{\Br}{Br}

\usepackage{color}

\title{The elliptic sieve and Brauer groups}

\author[S. Bhakta]{Subham Bhakta}
\address{Subham Bhakta\\
	Mathematisches Institut \\
	Georg-August-Universität Göttingen\\  
	Bunsenstraße 3-5 \\
	D-37073 Göttingen \\
	Germany}
\email{subham.bhakta@mathematik.uni-goettingen.de}

\author[D.Loughran]{Daniel Loughran}
\address{Daniel Loughran\\
	Department of Mathematical Sciences \\
	University of Bath \\
	Claverton Down \\
	Bath \\
	BA2 7AY \\
	UK.}
\urladdr{https://sites.google.com/site/danielloughran/}

\author[S. L. Rydin Myerson]{Simon L. Rydin Myerson}
\address{Simon L. Rydin Myerson\\
	Mathematics Institute \\
	Zeeman Building \\
	University of Warwick\\
	Coventry \\ CV4 7AL}
\email{simon.myerson@gmail.com}

\author[M. Nakahara]{Masahiro Nakahara}
\address{Masahiro Nakahara \\
	Department of Mathematics\\
	University of Washington\\
	Seattle\\
	WA 98195\\
	USA.}
\email{mn75@uw.edu}
\urladdr{https://sites.math.washington.edu/~mn75/index.html}

\begin{document}


	\begin{abstract}
		A theorem of Serre states that almost all plane conics over $\Q$
		have no rational point. We prove an analogue of this  for families
		of conics parametrised by elliptic curves using elliptic divisibility sequences
		and a version of the Selberg sieve for elliptic curves. We also give more general results
		for specialisations of Brauer groups, which yields applications to norm form equations.
	\end{abstract}
	
	\maketitle
	\setcounter{page}{1}
	\tableofcontents
	
	\section{Introduction}
	
	
	\subsection{Sums of two squares}
	A famous theorem of Landau and Ramanujan states that almost all integers are not sums of two squares, when ordered by absolute value. 
	In this paper we prove a version of this result for elliptic curves.
	
	\begin{theorem} \label{thm:y}
		Let $E$ be an elliptic curve over $\Q$ given by an integral Weierstrass equation.
		Let $P \in E(\Q)$ have infinite order with $P \in E(\R)^0$.
		Then there exists $\omega=\omega(E,P) > 0$ such that 
		\begin{equation}\label{eqn:thm-y}
			\#\{n \in \Z  \colon |n| \leq B, y(nP) \text{ is a sum of two squares} \} \ll_{E,P} B/(\log B)^{\omega}.
		\end{equation}
	\end{theorem}
	
	Here \(E(\R)^0\) denotes the connected component of the identity of $E(\R)$, and $y(nP)$ denotes the $y$-coordinate of the point $nP$; this is a rational number and we are asking that this is the sum of two rational squares. The result shows that for almost all multiples of $P$, the $y$-coordinate is not a sum of two (rational) squares. (See \S~\ref{sec:lower_bounds} for a discussion on sharpness of this result.)
	
	The reader may  wonder: why consider the $y$-coordinate and not the $x$-coordinate? Well, the corresponding result is \emph{false} otherwise: Consider the elliptic curve
	\begin{equation} \label{eqn:bad}
		y^2=x^3-k^2.
	\end{equation}
	Evidently \(x=(y^2+k^2)/x^2\) is always a sum of two squares.
	Our methods are able to handle the \(x\)-coordinate problem, but only with additional assumptions on $E$ (see Theorem~\ref{thm:x}).
	
	The assumption that $P \in E(\R)^0$  is slightly deeper and is intimately connected with our method; however without it we can obtain counter-examples to similar statements (see Example~\ref{ex:counter-example}).

	\subsection{Conic bundles}
	To understand exactly what is happening, we put our results into the more geometric framework of \emph{conic bundles}.
	In an influential paper \cite{Ser90}, Serre proved that almost all plane conics over $\Q$ have no rational point, when ordered by the size of their coefficients. This was a special case of a more general result \cite[Thm.~2]{Ser90} on conic bundles $\pi:X \to \P^n_\Q$, which says that providing $\pi$ has no rational section we have
	\begin{equation} \label{eqn:Serre}
		\#\{ x \in \P^n(\Q) : H(x) \leq B, x \in \pi(X(\Q)) \} \ll_{X,E,P} B^{n+1}/(\log B)^{\Delta(\pi)}
	\end{equation}
	for some $\Delta(\pi) > 0$, where $H$ denotes the usual naive height on projective space. Here by a conic bundle, we mean a surjective morphism of varieties all of whose fibres are isomorphic to plane conics. This result was generalised in \cite[Thm.~1.2]{ls16} to  other  families of varieties over $\P^n$.
	
	One of the aims of the paper is to obtain a version of Serre's result for conic bundles over elliptic curves. Here the crucial concept is that of a \emph{non-split fibre}: this is an irreducible  fibre  isomorphic to $2$ lines over a quadratic extension (called the splitting field of the fibre). The relevant conic bundle for Theorem~\ref{thm:y} is 
	\begin{equation} \label{eqn:y}
		x_1^2 + x_2^2 = y x_0^2,
	\end{equation}
	where the equation is an affine patch of the surface in $\P^2 \times E$.
	Here a non-split fibre occurs over the point at infinity, but in the example \eqref{eqn:bad} one can check that every fibre of the relevant conic bundle is split. So to get savings one requires at least a non-split fibre; this is essentially the content of our next theorem.
	
	\begin{theorem} \label{thm:intro}
		Let $E$ be an elliptic curve over $\Q$ given by an integral Weierstrass equation
		and $\pi: X \to E$ a non-singular conic bundle.
		Let $P \in E(\Q)$ have infinite order with $P \in E(\R)^0$.
		Assume that $\pi^{-1}(mP)$ is non-split with imaginary quadratic splitting field,
		for some $m \in \Z$.
		Then there exists $\omega=\omega(X,E,P) > 0$ such that
		$$\#\{n \in \Z  \colon |n| \leq B, nP \in \pi(X(\Q))\} \ll_{X,E,P} B/(\log B)^{\omega}.$$
	\end{theorem}

	The theorem says that under some technical assumptions, for almost all multiples of a given non-torsion rational point the associated conic has no rational point.
	The assumption $P \in E(\R)^0$ may look artificial, however it is  \textit{necessary} for the conclusion. 
	
	\begin{example} \label{ex:counter-example}
		Consider the elliptic curve with the point $P$ of infinite order
		$$E: y^2 = x(x+2)(x-3), \quad P = (-1,2)$$
		and  the conic bundle
		$$t_0^2 + t_1^2 = (x - x_1)(x - x_3)t_2^2$$
		where $x_i$ denotes the $x$-coordinate of $iP$.
		This has non-split fibres over $ \pm P$ and $\pm 3P$.
		Here $P \notin E(\R)^0$ so the assumptions of Theorem~\ref{thm:intro}
		don't hold.
		
		We claim that the fibre over every
		\textit{even} multiple of $P$ contains a rational point, so the conclusion
		of Theorem~\ref{thm:intro} in fact does not hold either.
		This is a special case of a more
		general construction (Proposition~\ref{prop:ramified}), but we explain the key ideas
		here.
		
		Firstly one checks that the fibre over $O$ has a rational point.
		So let $2nP$ be a non-trivial even multiple of $P$.
		It is clear there are always $p$-adic points
		for $p \equiv 1 \bmod 4$.
		We will show that for 
		every prime 	$p \equiv 3 \bmod 4$, we have $v_p((x-x_1)(x-x_3)) = 0$,
		which implies that the fibre has a $\Q_p$-point.
		Moreover as every element of $2 \Z P$ lies in
		the real component  of the identity it satisfies $x \geq 3$. Hence $(x - x_1)(x - x_3) > 0$
		so the conic has a real point. Hilbert's version of quadratic reciprocity now
		shows that every conic has a $\Q_2$-point, hence has a $\Q$-point by Hasse-Minkowski.

		So assume for a contradiction that there is some even multiple $2nP$ and
		$p \equiv 3 \mod 4$ such that $v_p((x-x_1)(x-x_3)) > 0$. (One can check that
		$v_p(x_3) \geq 0$ for all $p \equiv 3 \bmod 4$, so the valuation cannot be negative.)
		Suppose for example that $v_p(x  - x_1) > 0$, so that $2nP \equiv \pm P \bmod p$.
		Then $P$ is divisible by $2$ modulo $p$. However our example was chosen
		so that $2$-division field of $P$ is $\Q(i)$ (this can be shown using the criterion
		from \cite[\S2]{BZ18}), and since $p$ is inert in
		$\Q(i)$ it follows that $P$ is \textit{not} $2$-divisible modulo $p$; a contradiction.
		The case $2nP \equiv \pm 3P \bmod p$ is analogous.
	\end{example}
	
\subsection{Proof ingredients}\label{sec:proof}
Our key tools are sieves and elliptic divisibility sequences.

	In the classical sieve setting one usually sieves with respect to the homomorphisms $\Z \to \Z/p\Z$, for primes $p$, or more general prime powers. We originally tried to mimic this setting by sieving with respect to the homomorphisms $E(\Q) \to E(\F_p)$,  inspired by Kowalski's elliptic sieve \cite[\S 4.4]{kow08} used to study prime divisors in elliptic divisibility sequences. However our method quickly diverges from the classical setting and Kowalski's setting, as in our case information modulo $p$ is insufficient. A significant technical step in our proof is trying to control the $p$-adic valuations of the rational points we are sieving, which we achieve by  sieving modulo $p^{n_p}$ for varying primes $p$ and growing exponents $n_p$, so our sieve has no classical analogue. This difficulty is related to the fact that elliptic curves do not satisfy weak approximation, and does not arise in the classical sieve setting where $p$-adic valuations are easy to control. Our exact valuation theoretic problems are closely related to $p$	being a ``non-Wieferich prime for base $P\in E$'' in the sense of Voloch~\cite{V00}, and it is not even known whether there 	exists a single elliptic curve with infinitely many such primes \cite{Sil88}. These issues greatly complicate our sieve set up, and we have to work with the filtration structure on $E(\Q_p)$  to control valuations. (See Remark~\ref{rem:odd_valuation}.)
	
	To come up with a sieve criterion we  use elliptic divisibility sequences: these are  defined via the denominators of the $x$-coordinates of the multiples of $P$, multiplied by a sign to obtain better recurrence properties. We recall the relevant definitions in \S\ref{sec:EDS}. The key property for us is that elliptic divisibility sequences are periodic modulo an arbitrary integer (Proposition~\ref{prop:actually-periodic}). We achieve this using work  of Verzobio~\cite{Ver21} and does not seem to have been proven  in the literature before in this generality. In our proof we also have to be careful with signs, which requires us to use the work \cite{SS06} as well as equidistibution results for multiples of irrational numbers modulo $1$. Whilst these sign issues may seem a mere technical step, in fact they are crucial and related to our necessary assumptions on the real components of $E(\R)$ (cf.~Example~\ref{ex:counter-example}).

	Theorem~\ref{thm:intro} is a quantitative strengthening of a result of the fourth-named author and Berg \cite{bn19}, which proves under suitable assumptions that for certain conic bundles $\pi:X \to E$, the image $\pi(X(\Q))$ does not contain a translate of a subgroup of finite index.  In \cite{bn19} the authors only consider  elliptic curves which are Galois general in a sense captured by conditions (1)-(4) in their Theorem~3.5, and their results only apply to special conic bundles given by pulling back Ch\^{a}telet surfaces from $\P^1$ which also have a non-split fibre over a rational point. Our results apply to an overlapping collection of elliptic curves and conic bundles, but the key point  is that our conclusion is stronger:  a subset of $\Z$ which contains no arithmetic progression may still have positive density (e.g.~the set of squarefree numbers in $\Z$).

\subsection{Lower bounds and counting by height} \label{sec:lower_bounds}
		
	We believe that our paper demonstrates the usefulness of sieve techniques and elliptic divisibility sequences to counting problems on elliptic curves. The proof of Theorem~\ref{thm:intro} gives an explicit value for $\omega$, but we doubt that our upper bound is sharp. 
	Proving any kind of lower bound seems very difficult in general; we are only able to do this in various trivial cases where $\omega = 0$, so that $\pi(X(\Q))$ has positive density in $E(\Q)$ (see \S\ref{sec:examples}). The following question seems quite challenging.
	
	\begin{question}\label{q:infinite}
		Does there exist an elliptic curve $E$ over $\Q$ such that the set
		\begin{equation}\label{eqn:infinite}
			\{(x,y) \in E(\Q) \colon y \text{ is a sum of two squares} \}
		\end{equation}
		is infinite?
	\end{question}
	
	Standard conjectures in arithmetic geometry  seem to have nothing to say about this question, as the associated conic bundle surface is neither rationally connected nor of general type (over $\Q(i)$ it is birational to $\P^1 \times E$).
	
	However, the following heuristic suggests \eqref{eqn:infinite} should be quite sparse. Let  \(E(\Q)\) have rank \(r\). Fix a norm \(\|\,\cdot\,\|\) on \(\R^r\). The numerator and denominator of {\(y( n_1P_1+\dotsb+ n_r P_r)\) are both integers of  size  \( \exp(O_{E,P}(\|\vec{n}\|^2))\)}, with the denominator being a perfect cube. A proportion \(1/\|\vec{n}\|^2\) of such rational numbers are sums of two squares. One might therefore  {speculate that  the set \eqref{eqn:infinite} has size \( \ll \sum_{\vec{n}\in\N^r} 1/\|\vec{n}\|^2 \), so can be infinite only when \(r>1\).}
	
	Versions of this problem were raised by Poonen~\cite[Questions~23, 33]{AIM02} and Browning~\cite[Problem~10, pp3181-2]{Bro19}. Browning in particular asked about the number of points for which the denominator of \(y\) is a sum of two squares. A similar heuristic suggests the number of points \(\{nP: n \leq B\}\) with this property might be around \(\sum_{n\leq B}1/n \sim \log B\). Our methods give upper bounds for problem  without alteration.
		
	One can rephrase our results in terms of the \textit{canonical height} $\widehat{h}$ on $E$, as $nP$ has height roughly $n^2$. In this language the heuristic just discussed suggests the following.
\begin{align*}
	\MoveEqLeft[15]
	\#\{Q \in E(\Q):\widehat{h}(Q) \leq H,  y(Q) \text{ is a sum of two squares} \}
\\
&	\asymp \begin{cases}
	1,& \text{if }\rank E(\Q)\leq 1,
	\\
	\log H, & \text{if }\rank E(\Q)=2,
	\\
	H^{(r-2)/2}, & \text{if }\rank E(\Q)>2;
	\end{cases}
	\\
	\MoveEqLeft[15]
	\#\{Q \in E(\Q):\widehat{h}(Q) \leq H,  \text{the denominator of } y(Q) \text{ is a sum of two squares} \}
	\\&	\asymp \begin{cases}
	\log H, & \text{if }\rank E(\Q) = 1,
	\\
	H^{(r-1)/2}, & \text{if }\rank E(\Q)\geq 2.
	\end{cases}
\end{align*}

	Our present method cannot handle the case of rank \(>1\); the key stumbling block is that we have no control over the prime $p$ constructed in Proposition~\ref{prop:odd_valuation}, which would be necessary to combine $p$-adic information at sums of points. We are however able to prove  non-trivial upper bounds provided the curve has rank $1$. Rather than stating the most general result in terms of conic bundles, we content ourselves with the following variant of Theorem~\ref{thm:y}. 
	
	\begin{theorem} \label{thm:rank_1}
		Let $E$ be an elliptic curve over $\Q$ given by an integral Weierstrass equation.
		Assume that $E$ has rank $1$. Let 
		\[
		\eta = \begin{cases}
		1&\text{ if }\# E(\Q)^{\tors}\text{ has two distinct prime divisors,}
		\\
		-1 &\text{ otherwise.}
		\end{cases}
		\]
		Then
		\begin{multline*}
			\#\left\{Q \in E(\Q) \cap E(\R)^0 :
			\begin{array}{l}
				\widehat{h}(Q) \leq H, y(Q) \text{ is a }\\
				\text{sum of two squares} 
			\end{array}
			\right\}
			\ll_E 
			\frac{H^{1/2}(\log\log \log H)^{\eta/2}}{(\log\log H)^{1/2}}.
\end{multline*}
	\end{theorem}
	
	The total number of rational points in \( E(\R)^0\) of height at most $H$ is $\gg H^{1/2}$, since \(2E(\Q)\subset E(\R)^0\), so the theorem indeed shows that $0\%$ of these have $y$-coordinate which is a sum of two squares. Note that if \(E(\R)\) is connected, then the upper bounds applies to all rational points on $E$.

	\subsection{Generalisation to Brauer groups}
	We now state our most general result, of which Theorem~\ref{thm:intro} is a special case. Firstly, we are able to prove results for conic bundles whose non-split fibres have real quadratic splitting field. Secondly, we make explicit the dependence of  $\omega$ and the leading constant on $P$. Thirdly, our methods are sufficiently robust that they  allow applications to specialisations of Brauer group elements on elliptic curves. This is also the viewpoint taken by Serre in his paper \cite{Ser90}, as well in the more recent papers \cite{Lou13,ls16}. Brauer groups are formally easier to work with than conic bundles, since one does not require explicit equations and one can make use of Grothendieck's residue map. 
	Here we take a Brauer group element which is ramified at a rational point. The ramification gives rise to a cyclic extension of $\Q$ to which we associate a Dirichlet character using Kronecker--Weber (see \S\ref{sec:denom} for details and relevant background on Brauer groups). Our result here is as follows. 
	
	\begin{theorem} \label{thm:Brauer_intro}
		Let $E$ be an elliptic curve over $\Q$ given by an integral Weierstrass equation. Let $P \in E(\Q)$ have infinite order. Assume there is \(m\in\Z\) such that $b \in \Br \Q(E)$ is ramified at \(mP\) and let the associated 
		Dirichlet character $\chi$ have modulus \(q(\chi)\).
		Let $\beta_n$ be the elliptic divisibility sequence associated to $P$ and
		let $\pi$ be
		the period of  the sequence $\beta_n \bmod q(\chi)$.
		
		Assume that there is
		some index \(\alpha\in\N\) with $\gcd(\alpha,\pi) = 1$ which  satisfies either
		\begin{enumerate}
			\item  $\chi(|\beta_\alpha|) \neq 0,1$, or\label{itm:nonresidue-I}
			\item  $\chi(-|\beta_\alpha|) \neq 0,1$ and  $P \in E(\R)^0$, or \label{itm:nonresidue-II}
			\item $\chi(-|\beta_\alpha|) \neq 0,1$ and $4 \nmid \pi$.\label{itm:nonresidue-III}
		\end{enumerate}
		Then we have \(\pi\ll_{E,\chi} 1\) and
		$$
		\#\{|n| \leq B : b(nP) = 0 \in \Br \Q\} \leq 
		C_{E,P,b}
		\frac{B\log\log B}{
			(\log B)^{1/2\varphi(\pi)}},
		$$
		where  $\varphi$ denotes Euler's totient function and \(C_{E,P,b}\) is a positive constant depending  on \(E,P\) and \(b\) only and given by \eqref{eqn:constant} below.
	\end{theorem}

	In the statement $b(Q) \in \Br \Q$ denotes the evaluation of the Brauer element $b$ at $Q$. We abuse notation slightly and implicitly ignore the finitely many points where $b$ is not defined.
	
	The Brauer group framework essentially allows us to replace quadratic extensions by arbitrary cyclic extensions and conic bundles by higher-dimensional Brauer--Severi varieties. Moreover we can even handle some non-abelian extensions. As an example application in the style of Theorem~\ref{thm:y}, we have the following.
	
	\begin{theorem} \label{thm:non-abelian}
		Let $K/\Q$ be a number field which contains a cyclic subfield which is not totally real.
		Let $E$ be an elliptic curve over $\Q$ given by an integral Weierstrass equation and let \(P\in E(\Q)\) be a point of infinite order with $P \in E(\R)^0$.
		Then there exists $\omega=\omega(E,K)> 0$ such that 
		$$\#\{|n| \leq B \colon y(nP) \text{ is a norm from } K \} \ll_{E,P,K}
		B/(\log B)^{\omega},$$
		where the implicit constant is \(\exp( O_{E,K}({\widehat{h}(P)}/{\log \widehat{h}(P)})) \).
	\end{theorem}
	
	The fields $K = \Q(\sqrt[n]{a}, \mu_n)$ satisfy these hypotheses, for $a \in \Q^\times$ and $n \geq 3$. The explicit dependence on the point \(P\) is included as it is needed to prove Theorem~\ref{thm:rank_1}. (A version of Theorem~\ref{thm:rank_1} replaced with the condition that \(y(Q)\) is a norm from \(K\), where \(K\) is of the type in Theorem~\ref{thm:non-abelian}, would follow by the same proof.)
	
	Let us consider the technical assumptions on elliptic divisibility sequences in Theorem~\ref{thm:Brauer_intro}. We suspect that Condition \eqref{itm:nonresidue-I} holds for all but finitely many $\chi$ of given order; but it seems incredibly difficult to prove this, and even the analogous statement for the much simpler cases of Fibonacci or Mersenne numbers seems to be an open problem \cite{LM,Sun}. However if $|\beta_\alpha|$ is a non-square (say), which indeed holds for all but finitely many $\alpha$ \cite[Thm.~1.1]{ERS07}, then $\chi(|\beta_\alpha|) \neq 0,1$ for a positive proportion of quadratic Dirichlet characters; the challenge is it show that these characters cover all but finitely many as $\alpha$ varies. We are able to show the modest result that Condition \eqref{itm:nonresidue-I} holds $100\%$ of the time under suitable assumptions; see \S\ref{sec:100} for details. 
	
	As for our applications, if $\chi$ is an odd Dirichlet character  then we simply use that $\beta_1 = 1$ and $\chi(-1) = -1 \neq 0,1$ to see that Conditions \eqref{itm:nonresidue-II} or \eqref{itm:nonresidue-III} are satisfied. This is what makes stating Theorem~\ref{thm:intro} so simple, as for an imaginary quadratic extension $\Q(\sqrt{D})$ with $D$ a fundamental discriminant, the associated Dirichlet character is simply the Kronecker symbol $\left(\frac{D}{\cdot}\right)$, which takes the value $-1$ at $-1$ as $D$ is negative.

	We finish by returning to Example~\ref{ex:counter-example}.
	
	\begin{example} \label{ex:crazy}
		Consider the elliptic curve from Example~\ref{ex:counter-example}
		$$E: y^2 = x(x+2)(x-3)$$
		with $P = (-1,2)$.
		As we have seen, this pair does not satisfy the conclusion of Theorem~\ref{thm:intro}
		with respect to $\Q(i)$. Let us verify that the hypotheses
		of Theorem~\ref{thm:Brauer_intro} do not hold in this case.
		Firstly $P \notin E(\R)^0$ so Condition \eqref{itm:nonresidue-II} does not hold.
		Here $P$ has everywhere good reduction and the elliptic divisibility sequence $\beta_n$ is
		$$0,
		1,
		4,
		-65,
		-504,
		242369,
		-58888180,
		-66048490369,
		60955459632144, \dots$$
		The relevant Dirichlet character is the unique non-principal character $\chi$ modulo $4$.
		The sequence $\beta_n \bmod 4$ is just 
		$$0,1, 0, 3, 0, 1, 0, 3, 0, 1, 0, 3, 0, 1, 0, 3, 0, 1, 0, 3, 0,\ldots$$
		which has period $4$. Thus Condition \eqref{itm:nonresidue-III} from Theorem~\ref{thm:Brauer_intro}
		does not hold. Finally, the sequence $|\beta_n| \bmod 4$ is 
		$$0,1, 0, 1, 0, 1, 0, 1, 0, 1, 0, 1, 0, 1, 0, 1, 0, 1, 0, 1, 0,\ldots$$
		so \eqref{itm:nonresidue-I} does not hold either.
		In particular, this demonstrates that all hypotheses in Theorem~\ref{thm:Brauer_intro}
		are necessary in general for the conclusion to hold.
	\end{example}

	\subsection*{Outline of the paper}
	In \S\ref{sec:Qp} we recall various facts about elliptic curves over $\Q_p$. The following \S\ref{sec:EDS} contains a detailed study of elliptic divisilibity sequences. Here we prove periodicity modulo an arbitrary integer, and show our main technical result (Proposition~\ref{prop:odd_valuation}) on such sequences.  In \S\ref{sec:denom} we prove the theorems from the introduction. We give various examples illustrating our results in \S\ref{sec:examples}, as well as further examples which demonstrate that the conclusion of Theorem~\ref{thm:intro} does not hold for arbitrary conic bundles over elliptic curves. We finish in \S\ref{sec:100} by showing that the technical assumption in Theorem~\ref{thm:Brauer_intro} holds for $100\%$ of suitable Dirichlet characters of prime moduli.
	
	\subsection{Notation and conventions} \label{sec:not}
	We choose an embedding $\Q/\Z \subset \C^\times$. This corresponds to a choice of a compatible system of $n$th roots of unity for all $n$. We denote by $O= (0:1:0) \in \P^2$. By a Weierstrass equation with coefficients \(a_i\) we mean
	\begin{equation}\label{eqn:weierstrass}
		y^2+a_1xy+a_3y=x^3+a_2x^2+a_4x+a_6.
	\end{equation} 
	We denote by $\widehat{h}$  the canonical height on \(E(\Q)\) \cite[Ch.~VIII.9]{Sil09}. This extends to  a positive definitive quadratic form on $E(\Q) \otimes \R$ \cite[Prop.~VIII.9.6]{Sil09}. By \(O_{a,\dotsc,z}(A)\) we mean a quantity with absolute value at most \(CA\) for some positive  constant \(C\) depending on \(a,\dotsc,z\) only; if the subscripts are omitted the implied constant is absolute. We write \(A\ll_{a,\dotsc,z}B\) for \(A=O_{a,\dotsc,z}(B)\) and \(A=o(B)\) for \(A/B\to 0\).

	\subsection*{Acknowledgements}
	We thank Gergely Harcos and Efthymios Sofos for helpful comments and advice on some of the proofs. We also thank the anonymous referee for a careful reading of the paper which led to an improvement of Theorem 1.5.
	
	D.\@ Loughran and M.\@ Nakahara were sponsored by EPSRC grant EP/R021422/2.
	
	S.\@ Bhakta and S.\@ L.\@ Rydin Myerson were supported by the European Research Council (ERC) under the European Union's Horizon 2020 research and innovation program (grant ID 648329). S.\@ L.\@ Rydin Myerson was supported by DFG project number 255083470, and by a Leverhulme Early Career Fellowship.
	
	We thank the ZORP online seminar and its organisers. Initially there were two separate teams working on this problem using similar methods. But during a gathertown meeting on ZORP on 11th December 2020, we became aware of each other during discussions with Efthymios Sofos. Each team had a slightly different viewpoint and results, so we decided to combine to create, in our view, an ultimately superior paper.
	
	\section{Elliptic curves over \texorpdfstring{$\mathbb{Q}_p$}{Qp}} \label{sec:Qp}

	In this section let $E$ be an elliptic curve over $\Q_p$ given by a (not necessarily minimal) Weierstrass equation with coefficients in \(\Z_p\).
	Denote by $E_0(\Q_p)$ the set of points of $E(\Q_p)$ with non-singular reduction modulo $p$.
	We say that $P \in E(\Q_p)$ has \emph{bad reduction} if $P \notin E_0(\Q_p)$.
	There is a subgroup filtration
	\[
	\cdots \subset E_2(\Q_p) \subset E_1(\Q_p) \subset E_0(\Q_p),
	\quad E_i(\Q_p) = \{ P \in E_0(\Q_p) : P \equiv O \bmod p^i\}, \, i \geq 1.
	\]
	\begin{definition}\label{def:v_p}
		If \(P\in E(\Q_p)\setminus E_1(\Q_p)\) we set \(v_p(P)=0\).
		If $P \in E_1(\Q_p)$ we define
		$$v_p( P) = \sup \{ i : P \in E_i(\Q_p) \}.$$
	\end{definition}
	
	\begin{definition} \label{def:ord_p}
		For $P \in E_0(\Q_p)$ and $k \in \N$ we denote by $P \bmod p^k$
		the image of $P$ in $E_0(\Q_p) / E_k(\Q_p)$. We denote by
		$\ord(P \bmod p^k)$ its order.
	\end{definition}
	
	\begin{lemma} \label{lem:denominator_p_adic}
		Let $P =(x,y) \in E(\Q_p)$. Then 
		$v_p(P) = \max\{0,-v_p(x)/2\}.$
	\end{lemma}
	\begin{proof}
		If $v_p(x) \geq 0$ then $v_p(P) = 0$ so the result holds. So assume $v_p(x) < 0$.
		As the rational function $x/y$ is a uniformising parameter at $O$,
		we find that $v_p(P) = v_p(x/y)$. However, using $v_p(x) < 0$ and the Weierstrass
		equation, one finds that $2v_p(y) = 3v_p(x)$, and the claim easily follows.
	\end{proof}

	Lemma~\ref{lem:denominator_p_adic} gives a more explicit definition of the filtration which is often used in texts (e.g.~\cite[Ex.~VII.7.4]{Sil09}). We have the following inequality for the valuation of a multiple of a point.

	\begin{lemma} \label{lem:formal-group}
		Let $P \in E_1(\Q_p)$. Then $v_p(nP)\geq v_p(P) + v_p(n)$, with equality
		if $p \nmid n$.
	\end{lemma}
	
	\begin{proof}
		Hensel's lemma \cite[Lem.~2.1]{BL19} shows that 
		$|E_{i}(\Q_p)/E_{i+1}(\Q_p)| = p$ for all $i \geq 1$, thus this quotient
		is isomorphic to $\Z/p\Z$. The result now easily follows.
	\end{proof}

	\begin{remark}	 \label{rem:2}
		Using the formal group law on $E$ \cite[Thm.~IV.6.4(b), Prop.~VII.2.2]{Sil09},
		one can show that equality in fact holds except possibly if $p=2, v_p(P) = 1$ and \(p\mid n\).
		(See also \cite[Thm.~3]{Sta06} for a version over number fields.)
		The hypothesis is required for $p=2$. Take
		$$E: y^2 +xy = x^3 + 4x + 1, \quad P = (15/4, -83/8).$$
		Then $v_2(P) = 1$, but one calculates that $v_2(2P) = 4$.
		(The issue here is that $E_1(\Q_2)$ has non-trivial $2$-torsion,
		so is not isomorphic to $2\Z_2$.)
	\end{remark}

	\section{Elliptic divisibility sequences} \label{sec:EDS}
	
	\subsection{Basic properties}
	Now let $E/\Q$ be an elliptic curve given by a Weierstrass equation \eqref{eqn:weierstrass} with coefficients \(a_i\in\Z\). Let $P \in E(\Q)$ be a non-torsion point. Throughout this section we consider $E$ and $P$ as being fixed. 
	
	For any integer $n\geq0$, define the $n$th division polynomial $\psi_n\in\Z[x,y]$ as follows.
	\begin{align*}
		&\psi_0=0,\quad \psi_1=1,\quad \psi_2=2y + a_1x + a_3,\\
		&\psi_3=3x^4 + b_2x^3 + 3b_4x^2 + 3b_6x + b_8,\\
		&\psi_4=\psi_2(2x^6 + b_2x^5 + 5b_4x^4 + 10b_6x^3 + 10b_8x^2+(b_2b_8 − b_4b_6)x + b_4b_8 − b_2^6)
	\end{align*}
	where the $b_i$ are defined in \cite[Chapter III]{Sil09}, with subsequent polynomials given by
	\begin{equation}
		\begin{split}
			\psi_{2n+1}=\psi_{n+2}\psi_n^3-\psi_{n+1}^3\psi_{n-1},\quad n\geq2, \\
			\psi_{2n}\psi_2=\psi_n(\psi_{n+2}\psi_{n-1}^2-\psi_{n-2}\psi_{n+1}^2),\quad n\geq 3,
		\end{split}
	\end{equation}
	and extend this to negative \(n\) by setting \(\psi_n=-\psi_{-n}\). 
	These formulas are equivalent to the  recurrence relation
	\begin{equation}\label{eqn:elldivrec-1}
		\psi_{m+n}\psi_{m-n}\psi_r^2=\psi_{m+r}\psi_{m-r}\psi_n^2-\psi_{n+r}\psi_{n-r}\psi_m^2
	\end{equation}
	for any integers $m,n,r$. 
	The sequence $\psi_n$ forms a divisibility sequence in \(\Z[x,y]\), i.e.~$\psi_n\mid \psi_m$ for $n\mid m$. One notion of an \emph{elliptic divisibility sequence (EDS)} in a commutative ring would be a divisibility sequence satisfying \eqref{eqn:elldivrec-1}. The study of  EDS in \(\Z\), in this sense, was begun by Ward~\cite{War48}, and a modern exposition can be found in \cite[Ch.~10]{EPSW}. We will use a slightly different kind of EDS considered by Verzobio~\cite{Ver21}, which is better suited to our purpose.
	
	We can interpret \(x,y\) and each \(\psi_n\) as rational functions on \(E(\Q)\). By  \cite[Ex.~III.3.7]{Sil09}, multiplication by \(n\) is given as a rational map by 
	\[
	[n](P)
	=
	\left(
	\frac{x(P)\psi_n^2-\psi_{n-1}\psi_{n+1}}{\psi_n^2},
	\frac{\psi_{n-1}^2\psi_{n+2}-\psi_{n-2}\psi_{n+1}^2}{4y(P)\psi_n^3}
	\right).
	\]
	In particular $\psi_n$ is the square root of the denominator of the $x$--coordinate;
	the problem for us is that in general there may be some common factors between the numerator
	and denominator, so it will not be in lowest terms. We want to work with the genuine denominator as it has better $p$-adic properties (cf.~Lemma~\ref{lem:denominator_p_adic}).

	\begin{definition}\label{def:EDSB}
		Define the sequence $e_n$ by
		$nP=(a_n/e_n^2,b_n/e_n^3)$ with $\gcd(a_nb_n,e_n) = 1$ and $e_n > 0$. Writing \(\sign(t)=t/|t|\) for any \(t\neq 0\), set
		\[
		\beta_0 = 0, \qquad \beta_n = \sign(\psi_n(P)) \frac{e_n}{e_1}, 
		\qquad (n\in \Z\setminus\{ 0\}).
		\] 
	\end{definition}

	The sequence \(\beta_n\) is \textit{not} in general an elliptic divisibility sequence in the traditional sense, since it need not satisfy the recurrence relation \eqref{eqn:elldivrec-1}; differences can occur if $P$ admits  primes of bad reduction. In \cite{Ver21}  Verzobio calls such sequences EDSB, as opposed to sequences of the form \(\psi_n(P)\) which he terms EDSA. 
	He shows in~\cite[Thm.~1.9]{Ver21} that  the following  weakened version of \eqref{eqn:elldivrec-1} does hold for an EDSB.

	\begin{proposition}[Verzobio]\label{prop:Verzobio}
		Set \[M=M(P)=\lcm \{\ord(P + E_0(\Q_p)) : p \text{ prime}\},\]
		where  $\ord(P + E_0(\Q_p))$ denotes the order of the image of $P$ in the finite group $E(\Q_p)/E_0(\Q_p)$. 
		Let \(n, m, r\in\Z\) of which two are multiples of \(M(P)\). Then
		\begin{equation}\label{eqn:verzobio}
			\beta_{n+m}\beta_{n-m}\beta_r^2
			=
			\beta_{m+r}\beta_{m-r}\beta_n^2
			-\beta_{n+r}\beta_{n-r}\beta_m^2.
		\end{equation}
	\end{proposition}	
	\begin{remark}\label{rem:tamagawa}
		Here \(M\) is the least positive integer such that \(MP\) has everywhere good reduction. It divides \(\prod_p\#(E(\Q_p)/E_0(\Q_p))\), which is the product  of the Tamagawa numbers of $E$ if the model is globally minimal.
	\end{remark}

	Verzobio defines \(\beta_n\) for \(n\geq 0\) and proves the theorem under the assumption \(n\geq m\geq r>0\); in our notation this can be removed by using \(\beta_{-n}=-\beta_n\) and permuting the variables as appropriate.

	To illustrate some of the nice $p$-adic properties of this sequence, we prove that it is a \textit{strong divisibility sequence}. We first make explicit Lemma~\ref{lem:denominator_p_adic}.
	
	\begin{lemma} \label{lem:p-adic_beta}
		For all primes $p$ we have \(v_p(\beta_n) = v_p(nP) - v_p(P)\).
	\end{lemma}
	\begin{proof}
		Immediate from the definition and Lemma~\ref{lem:denominator_p_adic}.
	\end{proof}
	
	\begin{lemma}\label{lem:gcd}
		For all \(n,m\in\Z\) we have \(\gcd(\beta_m,\beta_n)=|\beta_{\gcd(m,n)}|\).
	\end{lemma}
	
	\begin{proof}
		By Lemma~\ref{lem:p-adic_beta}, for any prime \(p\) and any \(V\in\N\) we have
		\[
		\{n\in\Z: v_p(\beta_n) \geq V\}=\{n \in \Z :nP\in E_{V+v_p(P)}(\Q_p)\}=q\Z
		\]
		for some \(q\in\N\). 
		In particular
		\(
		p^V\mid \beta_n
		\) if and only if \(
		q\mid n
		\). Therefore
		\[
		p^V
		\mid 
		\gcd(\beta_m,\beta_n)
		\iff
		q\mid \gcd(m,n)
		\iff
		p^V \mid \beta_{\gcd(m,n)}.\qedhere
		\]
	\end{proof}
	
	We emphasise that an EDSA need not be a strong divisibility sequence if $P$ admits primes of bad reduction. The elegance of Verzobio's EDSB is that it has both good $p$-adic properties and comes within a whisker of satisfying the recurrence relation.

	\subsection{Symmetry law}
	
	A central part of Ward's work on elliptic divisibility sequences is a \emph{symmetry law}~\cite[Thm.~8.1]{War48} (see \cite[Thm.~1.11]{ABY} for a modern formulation). This says that an integral EDSA modulo a prime forms a periodic sequence of a certain form.  We  prove a version of this for EDSBs for general prime powers.

	\begin{proposition}\label{prop:symmetry-law}
		Let \(M\) be as in Proposition~\ref{prop:Verzobio}. 
		Let \(n,r\in\Z\) with \(M\mid r\). Let \(p\) be a prime and let \(k\in\N\). Suppose that \(p^k\) divides \(\beta_r/\gcd(\beta_r,\beta_M)\).
		Then for all \(\ell \in \Z\)  we have
		\[
		\beta_{n+\ell r}
		\equiv
		\begin{cases}
			\left(\beta_{M+r}\beta_{M-r}\beta_M^{-2}\right)^{\frac{\ell(\ell-1)}{2}}
			(\beta_{n+r}\beta_{n}^{-1})^\ell \beta_{n}
			\bmod p^k,
			&
			\text{if }p^k\nmid \beta_n,
			\\
			0
			\bmod p^k,
			&
			\text{if }p^k\mid \beta_n,
		\end{cases}
		\]
		where in the first case the quotients \(\beta_{M+r}\beta_{M-r}/\beta_M^2\) and \(\beta_{n+r}/\beta_n\) are \(p\)-adic units.
	\end{proposition}
	
	\begin{proof}
		Lemma~\ref{lem:gcd}  gives us
		\begin{equation}\label{eqn:gcds_of_n_r}
			|\beta_{\gcd(n,r)}|=	\gcd(\beta_{n+\ell r},\beta_r)=	\gcd(\beta_{n},\beta_r)
		\end{equation}
		for every \(\ell\in\Z\). This proves the proposition if \(p^k\mid \beta_n\), so assume that \(p^k\nmid \beta_n\).
		
		Taking $m = M$ in Proposition~\ref{prop:Verzobio}, and replacing \(n\) by \(n+\ell r\), we obtain
		\begin{equation}
			\label{eqn:verzobio-application}
			\beta_{M+r}\beta_{M-r}\beta_{n+\ell r}^2
			\equiv\beta_{n+(\ell+1)r}\beta_{n+(\ell-1)r}\beta_M^2
			\bmod \beta_r^2,
		\end{equation}
		for any \(\ell\in \Z\).
		We want to combine this with Lemma~\ref{lem:gcd}.
		Let
		\begin{align}
			C&=\frac{\beta_{M+r}\beta_{M-r}}{\beta_M^2},
			&
			a_\ell &=\frac{ \beta_{n+\ell r}}{\gcd(\beta_n,\beta_r)}.
			\label{eqn:def-of-C-and-a}
		\end{align}
		Since \(M\mid r\),	Lemma~\ref{lem:gcd} shows that  \(C\in\Z\). 
		Also \eqref{eqn:gcds_of_n_r} shows that \(a_\ell\) is an integer coprime to \(\beta_r/\gcd(\beta_n,\beta_r)\). 
		Hence, dividing  both sides of \eqref{eqn:verzobio-application} by \(\beta_{n+\ell r}^2\beta_M^2\) gives
		\begin{equation}\label{eqn:one-step-recurrence}
			C
			\equiv
			\frac{a_{\ell+1}a_{\ell-1}}{a_{\ell}^2}
			\bmod \frac{\beta_r^2}{\gcd(\beta_n\beta_M,\beta_r)^2}
			\quad\text{for all }\ell\in\Z,
		\end{equation}
		where every  \(a_{\ell}\) is coprime to the modulus. It follows by induction on \(\ell\) from \eqref{eqn:one-step-recurrence} that
		\[
		a_\ell
		\equiv
		C^{\frac{\ell(\ell-1)}{2}} a_1^\ell a_0^{1-\ell}
		\bmod\frac{\beta_r^2}{
			\gcd(\beta_n\beta_M,\beta_r)^2}.
		\]
		Multiplying by \(\gcd(\beta_n,\beta_r)\) we obtain
		\[
		a_\ell\gcd(\beta_n,\beta_r)
		\equiv
		C^{\frac{\ell(\ell-1)}{2}} (a_1a_0^{-1})^\ell a_0\gcd(\beta_n,\beta_r)
		\bmod\frac{\gcd(\beta_n,\beta_r)\beta_r^2}{
			\gcd(\beta_n\beta_M,\beta_r)^2}.
		\]
		Here \(\beta_r/\gcd(\beta_M,\beta_r)\) divides the modulus, and so the congruence holds modulo \(p^k\). Inserting the definitions \eqref{eqn:def-of-C-and-a} proves the first case in the proposition.
		
		Finally, since \(p^k\mid \beta_r/\gcd(\beta_M,\beta_r)\) and \(p^k\nmid \beta_n\), we see that \(p\) divides the modulus in \eqref{eqn:one-step-recurrence}. Since every \(a_\ell\) is coprime to the modulus, we see that \(C\) and \(\beta_{n+r}\beta_n^{-1}=a_1a_0^{-1}\) are \(p\)-adic units, as claimed in the final part of the proposition.
	\end{proof}
	
	\subsection{Periodicity}
	We now use the symmetry law to prove that $\beta_n$ is perodic modulo any prime power, and hence modulo any integer. Versions of this appear in the literature for differing definitions of EDS. Ward proved eventual periodicity modulo any prime in \cite[Thm.~11.1]{War48}. Shipsey proved a version modulo $p^2$ for primes of good reduction \cite[Thm.~3.5.4]{Shi00}. Ayad proved it modulo any integer, but assuming good reduction and avoiding $p=2$ or ``rank of apparition $2$'' \cite[Thm.~D]{Aya93}. Silverman proved a version over finite fields \cite[Thm.~1]{Sil05} as well as a version modulo prime powers whenever the curve has good ordinary reduction  \cite[Thm.~3]{Sil05}. Our version (Proposition~\ref{prop:actually-periodic}) contains none of these technical assumptions and is a general version of periodicity, for Verzobio's arguably more elegant EDSB. 
	
	Our result is the following, which  shows periodicity modulo an arbitrary prime power and gives an upper bound for the period. Note that the Chinese Remainder Theorem then easily shows periodicity modulo an arbitrary integer.

	\begin{proposition}\label{prop:actually-periodic}
		Let \(M\) be as in Proposition~\ref{prop:Verzobio}, let \(k\in\N\), and let \(p\) be a prime.
		Let
		\begin{equation}\label{eqn:def-of-r}
			r(p^k)= M\ord(MP\bmod p^{k+v_p(MP)})
		\end{equation}
		and
		\begin{equation}\label{eqn:def-of-period}
			\pi(p^k)
			=
			\begin{cases}
				(p-1)p^{k-1}r(p^k),&\text{if }p\neq 2\text{ and } \left(\frac{\beta_{M+r(p^k)}\beta_{M-r(p^k)}}{p}\right)=1,
				\\
				2(p-1)p^{k-1}r(p^k),&\text{otherwise}.
			\end{cases}
		\end{equation}
		Then for every $m \in \Z$ 
		we have
		\begin{equation*}\label{eqn:actually-periodic-conclusion}
			m\equiv n \bmod \pi(p^k) {\hskip.5em\relax}\implies{\hskip.5em\relax}
			\beta_m
			\equiv \beta_n \bmod p^k.
		\end{equation*}
		In  other words the sequence
		$\beta_m \bmod p^k$	is  periodic with period dividing  \(\pi(p^k) \).
	\end{proposition}
	
	We could slightly simplify the proof by defining \(\pi(p^k)=2(p-1)p^{k-1}r(p^k)\) in all cases. However it is of some interest to find cases in which \(4\nmid \pi(p^k)\), because this allows us to remove the condition \(P\in E(\R)^0\) in some of our results (see Theorem~\ref{thm:Brauer_intro}). This is our reason to include the first  case in \eqref{eqn:def-of-period}.
	
	\begin{proof}
		For ease of notation we write \(r=r(p^k)\) throughout the proof.
		
		We first observe that \(rP \equiv O \bmod p^{k+v_p(MP)}\) by \eqref{eqn:def-of-r}. That is  we have \(k+v_p(MP)\leq v_p(rP)\), and hence by Lemma~\ref{lem:p-adic_beta} and \eqref{eqn:def-of-r} we have
		\begin{equation}
			p^k \text{ divides }\frac{\beta_r}{
				\gcd(\beta_M,\beta_r)}
			\quad \text{ and }\quad
			M\mid r.
			\label{eqn:technical-condition-for-periodicity}
		\end{equation}
		Let \(n\in\Z\). By \eqref{eqn:technical-condition-for-periodicity}, the hypotheses of  Proposition~\ref{prop:symmetry-law} are satisfied.
		If \(p^k\mid\beta_n\) then the result follows immediately; suppose therefore that \(p^k\nmid\beta_n\).
		Proposition~\ref{prop:symmetry-law} shows that 
		\[
		\beta_{n+\ell r}
		\equiv
		\left(\beta_{M+r}\beta_{M-r}\beta_M^{-2}\right)^{\frac{\ell(\ell-1)}{2}}
		(\beta_{n+r}\beta_{n}^{-1})^\ell \beta_{n}
		\quad
		\bmod p^k,
		\]
		for every \(\ell\in\Z\), where  \(\beta_{M+r}\beta_{M-r}\beta_M^{-2},
		\beta_{n+r}\beta_{n}^{-1}\)  are  $p$-adic units. In particular \(\gcd(\beta_n,p^k)=\gcd(\beta_{n+\ell r},p^k)=\gcd(\beta_n,\beta_r,p^k)\).
		
		Now \(\#(\Z/p^k\Z)^\times=(p-1)p^{k-1}\), and so if \(u\in\Z_p^\times\) then
		\[
		2(p-1)p^{k-1}\mid \ell
		\implies
		u^{\frac{\ell(\ell-1)}{2}}  \equiv 1
		\bmod p^k.
		\]
		Moreover if \(p\neq 2\) and \(\left(\frac{u}{p}\right)=1\) then \(u=v^2\) for \(v\in \Z_p^\times\). So
		\[
		(p-1)p^{k-1}\mid \ell,\,p\neq 2, \,\left(\frac{u}{p}\right)=1
		\implies
		u^{\frac{\ell(\ell-1)}{2}} \equiv 1
		\bmod p^k
		.
		\]
		Thus by definition of \(\pi(p^k) \), if \(\pi(p^k) \mid \ell r\) then
		\[
		\left(\beta_{M+r}\beta_{M-r}\beta_M^{-2}\right)^{\frac{\ell(\ell-1)}{2}}
		(\beta_{n+r}\beta_{n}^{-1})^\ell 
		\equiv 1\bmod p^k,
		\]
		which implies
		\[
		\beta_{n+\ell r}
		\equiv \beta_{n}
		\bmod  p^k.
		\]
		Writing \(m=n+\ell r\) completes the proof.
	\end{proof}
	
	There is a simpler, but slightly weaker, bound for the period.
	
	\begin{lemma}\label{lem:period_minimal}
		Let \(M\) be as in Proposition~\ref{prop:Verzobio}, let \(k\in\N\), and let \(p\) be a
		prime. Then 
		the period of \(\beta_n\bmod p^k\)  divides
		\[
		\begin{cases}
			2M(p-1)p^{2(k-1)} \ord(MP \bmod p), &\text{if } v_p(MP) = 0, \\
			2M(p-1)p^{2k-1}, \quad &\text{otherwise}. \\		
		\end{cases}\]
	\end{lemma}
	\begin{proof}
		Let $Q = MP$. 
		By Proposition~\ref{prop:actually-periodic},
		it suffices to show that
		\[\ord(Q\bmod p^{k+v_p(Q)}) \,\text{  divides  }\, r_1(p^k) := 
		\begin{cases}
			p^{k-1} \ord(Q \bmod p), &\text{if } v_p(Q) = 0, \\
			p^{k}, \quad &\text{otherwise}. \\		
		\end{cases}\]	
		If $v_p(Q) = 0$ then Lemma~\ref{lem:formal-group} implies that
		$v_p(p^{k-1}\ord(Q \bmod p)Q) \geq k-1 + v_p(\ord(Q \bmod p)Q)
		\geq k.$ If $v_p(Q) > 0$ then Lemma~\ref{lem:formal-group}
		yields $v_p(p^kQ) \geq k + v_p(Q)$. In both cases 
		$r_1(p^k)Q \equiv 0 \bmod p^{k+ v_p(Q)}$, as required.
	\end{proof}
	
	\begin{remark}\label{rem:uniform-period}
		By definition $M$ divides $\prod_{p} |E(\Q_p)/E_0(\Q_p)|$,
		hence is bounded uniformly with respect to $P$. Moreover
		$\ord(MP \bmod p)$ divides $|E_0(\Q_p) / E_q(\Q_p)|$.
		Thus Lemma~\ref{lem:period_minimal} shows that the period of 
		$\beta_n \bmod N$ can be bounded independently of $P$ for all $N \in \N$,
		with the bound only depending on $E$ and $N$.
	\end{remark}
	
	\subsection{Signs} 
	
	Recall from Definition~\ref{def:EDSB} that the sign of \(\beta_n\) is the sign of the sequence \(\psi_n(P)\).
	The following is \cite[Thm.~4]{SS06} (see also \cite{AKY} for a generalisation.)
	
	\begin{proposition}[Silverman-Stephens]\label{prop:silverman-stephens}
		There is a sign \(\sigma \in\{\pm1\} \) and an irrational number $\beta$
		such that for  all $n \in \N$ we have
		\[
		\sigma^{n-1} \sign(\beta_n)=
		\begin{cases}
			(-1)^{\lfloor n\beta\rfloor},
			&\text{if}~P\in E(\mathbb{R})^{0},
			\\ (-1)^{\lfloor n\beta \rfloor+\frac{n}{2}},
			&\text{if}~P\notin E(\mathbb{R})^{0}~\text{and}~n~\text{is~even,}
			\\
			(-1)^{\frac{n-1}{2}},
			&\text{if}~P\notin E(\mathbb{R})^{0}~\text{and}~n~\text{is~odd.}
		\end{cases}
		\]
		If \(P\in E(\mathbb{R})^{0}\) then  \(\beta\) is defined as follows. We fix an \(\R\)-analytic group isomorphism \(\psi:E(\R)^0\to \R^\ast_{>0}/e^{\Z}\) . Then let \(\beta = \log u\) where \(u\) is a representative of \(\psi(P)\) in \(\R^\ast_{>0}\) with  \(e^{-1}<u<1\).
	\end{proposition}
	
	In Silverman and Stephens' original statement of the theorem there is an isomorphism \(E(\R)\to\R^\ast/q^\Z\), which maps \(E(\R)^0\) to either \(\R^\ast_{>0}/q^\Z\) if \(q>0\) or \(\R^\ast_{>0}/q^{2\Z}\) otherwise. Without loss of generality we can assume that \(E(\R)^0\)  is mapped to \(\R^\ast_{>0}/e^{\Z}\), or else we can compose  our isomorphism with \(v\mapsto v^{-1/\log q}\) or \(v^{-1/2\log q}\). When  \(P\in E(\mathbb{R})^{0}\) their choice of \(u\) then satisfies \(e^{-1}<u<1\) as above.

	We want to say something about the Diophantine approximation properties of the irrational number \(\beta\) from the theorem. 
	Let \(\exp_E:\C \to E(\C)\) be the usual parametrisation of \(E\) using the Weierstrass \(\wp\)-function, see for example~\cite[Cor.~VI.5.1.1]{Sil09}. Bosser and Gaudron~\cite[Thm.~1.2]{bg19} proved:
	\begin{proposition}[Bosser-Gaudron]\label{prop:bosser-gaudron}
		Let \(z\in \C\) such that \(\exp_E(z)\in E(\Q)\setminus\{O\}\). Then we have
		\begin{equation*}
			\log |z|
			\gg_E \textcolor{black}{-1}-\widehat{h}(\exp_E z),
		\end{equation*}
		where \(\widehat{h}\) is the canonical height, as in \S\ref{sec:not}.
	\end{proposition}
A remark about the definition of  \(\exp_E:\C \to E(\C)\) may be helpful. The usual convention would be to normalise this map so that in a certain sense the derivative of \(\exp_E\) at the origin is the identity. This is not necessary for our purposes, and the naive parametrisation familiar from a first course on elliptic curves suffices. We only use the fact that \(\exp_E\) is a fixed \(\R\)-analytic surjective additive group homomorphism, and  Proposition~\ref{prop:bosser-gaudron} which holds regardless of normalisation.
	We use these to prove
	\begin{lemma}\label{lem:diophantine}
		Suppose the point \(P\) from the start of this section satisfies  \(P\in E(\R)^0\). Let \(\beta\) be as in Proposition~\ref{prop:silverman-stephens}, and let \(N\in\Z\setminus\{0\}\). Then
		\[
		\min_{M\in\Z}\log |N\beta-M|
		\gg_E \textcolor{black}{-1}-
		\widehat{h}(NP).
		\]
	\end{lemma}
	
	\begin{proof}
		Let \(w\in\C^*\) such that \(\exp_E(w)=P\), so that \(\exp_E(w\R)=E(\R)^0\) and \(\psi(\exp_E(tw))=e^{t\beta+\Z}\) for any \(t\in \R\).
		For any \(M\in\Z\) we deduce that
		\[
		\exp_E(N+\beta^{-1}M) = \phi^{-1}(e^{N\beta+M+\Z}).
		\]
		By the definition of \(\beta\) we deduce \(\exp_E(N+\beta^{-1}M) = \phi^{-1}(u^Ne^\Z)\) which is \(NP\) by definition of \(u\). That is,
		\[
		\exp_E^{-1}(NP) \supseteq  \{ (N+\beta^{-1}M)w :  M\in\Z\}.
		\]
		Now by Proposition~\ref{prop:bosser-gaudron}, any \(t\) such that \(tw\in\exp_E^{-1}(NP)\) has \(\log |t|\gg_E \textcolor{black}{-1}-
		\widehat{h}(NP)\), and so
		\[
		\min_{M\in\Z}\log |N\beta-M|
		\gg_E \textcolor{black}{-1}-
		\widehat{h}(NP).\qedhere
		\]
	\end{proof}

	\subsection{Main result}
	We now provide the main technical input required for the results stated in the introduction (Proposition~\ref{prop:odd_valuation}). Under certain assumptions, it stipulates the existence of many prime-numbered elements of the sequence $\beta_n$ which are divisible by primes to a certain valuation that are non-trivial with respect to a given Dirichlet character. We require an effective version of uniform distribution modulo $1$ for primes in an arithmetic progression multiplied by an irrational. This is deduced from an  exponential sum estimate. To begin, we quote two standard results on exponential sums in primes from Vaughan~\cite[Thm.~3.1,  Lem.~3.1]{Va97}.
	\begin{lemma}[Vinogradov]\label{lem:vinogradov}
		If 
		\(\alpha\in \R, a\in\Z,q\in\N\) with \(\gcd(a,q)=1, q\leq y,\) and \(|\alpha-a/q| \leq q^{-2}\) then
		\[
		\sum_{\substack{p\leq y\\ p \text{ prime}}} (\log p) e(\alpha p)
		\ll
		(\log y)^4
		(y q^{-1/2}+y^{4/5}+y^{1/2}q^{1/2}).
		\]
	\end{lemma}
	\begin{lemma}[Siegel–Walfisz for linear exponential sums]\label{lem:major-arcs} Let \(B>0\).
		If \(\alpha\in \R, a\in\Z,q\in\N\) with \(\gcd(a,q)=1, q\leq (\log y)^B\) and \( |\alpha-a/q| \leq (\log y)^B/y\) then there is \(C_B>0\) such that
		\[
			\sum_{\substack{p\leq y\\ p \text{ prime}}}  e(\alpha p)
		=\frac{\mu(q)}{\varphi(q)}\sum_{m =1}^y e((\alpha-a/q)m)
		+O_B(y\exp(-C_B{\sqrt{\log y}}))
		\]
		where  \(\mu\) is the M\"obius function and \(\varphi\) is the Euler totient function.
	\end{lemma}

	We use these to estimate sums of the form \(\sum_{\ell\leq y,\ell\text{ prime}} (\log \ell) e(\alpha \ell)\).
	\begin{lemma}\label{lem:exp-sum-0}
		Suppose that we have \(\alpha\in\R, a\in\Z,q\in\N,y\in\R\) such that
		\begin{align}\label{eqn:conditions-on-a-q}
			\gcd(a,q)&=1,& q&\leq y
			&|\alpha-a/q| &\leq 1/qy.
		\end{align}
	Then
	\begin{equation}\label{eqn:exp-sum-primes-majorant-0}
		\sum_{\substack{\ell\leq y\\ \ell \text{prime}}} (\log \ell) e(\alpha \ell)
		\ll
		y(\log y)^{-1}
		+
		\frac{y\log\log q}{q}.
	\end{equation}
	\end{lemma}
\begin{proof}
	If \(q\leq (\log y)^5\) we apply Lemma~\ref{lem:major-arcs}; otherwise we apply Lemma~\ref{lem:vinogradov}.
	Recalling the standard bound \(\varphi(q)\gg \frac{q}{\log \log q}\) gives the result.
\end{proof}
	
	From this simple estimate we pass to a more difficult exponential sum.
	
	\begin{lemma}\label{lem:exp-sum}
		Let $s,t \in \N$ with $\gcd(s,t) = 1$ and let $\beta$ be as in Proposition~\ref{prop:silverman-stephens}.
		For all \(y\geq e^{e^e},j\in\N\) we have
		\begin{equation}\label{eqn:exp-sum-primes}
			\sum_{\substack{\ell \leq y \\ \ell \equiv s \bmod t\\ \ell \text{ prime}}}(\log \ell )
			e(j \ell \beta/2)
			\ll_E 
			ytj \sqrt{\frac{\widehat{h}(P)}{\log y}} \log\log\log y.
		\end{equation}
	\end{lemma}

		\begin{proof}
		Fix  \(y\geq 1,j\in\N\). 
		We use the formula
		$$ \frac{1}{t} \sum_{ m \in \Z/t\Z} e(m(n-s)/t) = 
		\begin{cases}
			1, & n \equiv s \bmod t, \\
			0, & \text{otherwise},
		\end{cases}
		$$
		which is valid for all integers $n$. This gives
		\begin{align*}
			\sum_{\substack{\ell \leq y \\ \ell \equiv s \bmod t}} (\log \ell)e(j \ell \beta/2)
			& = \frac{1}{t}  \sum_{m \in \Z/t\Z} e(-ms/t)	\sum_{\substack{\ell \leq y}} (\log \ell)
			e(( m/t + j \beta/2)\ell).	
		\end{align*}
			We apply Lemma~\ref{lem:exp-sum-0} to estimate the final sum in		
		 \eqref{eqn:exp-sum-primes}. We set
		 \begin{equation}
		 	\label{eqn:choice-of-alpha}\alpha = m/t+j\beta/2.
		\end{equation}
	We fix  \(a\in\Z,q\in\N\) satisfying \eqref{eqn:conditions-on-a-q}, noting that the existence of such \((a,q)\) is guaranteed by 
	 Dirichlet's approximation theorem. Then \eqref{eqn:exp-sum-primes-majorant-0} becomes
		\begin{equation}\label{eqn:exp-sum-primes-majorant}
			\sum_{\ell\leq y} (\log \ell) e((m/t+j\beta/2) \ell)
			\ll
			\frac{y}{\log y}
			+
			\frac{y\log\log q}{q}.
		\end{equation}
		For this to be useful we need to deduce from \eqref{eqn:conditions-on-a-q} and \eqref{eqn:choice-of-alpha} a lower bound on \(q\). We are going to prove that
		\begin{equation}\label{eqn:bound-on-q}
		\text{either}\quad y\leq t^{3/2}\quad\text{or}\quad
		q\geq \frac{1}{tj}\sqrt{\frac{\log y}{\widehat{h}(P)}}.
	\end{equation}
		
		We start with the final condition in \eqref{eqn:conditions-on-a-q}, and  multiply it by \(2tq\) to get
			\[\log(2t/y)\geq
			\min_{M\in\Z}\log |2tq\alpha-M|.
		\]
		The definition \eqref{eqn:choice-of-alpha} gives
		\[
		\min_{M\in\Z}\log |2tq\alpha-M|=
		\min_{M\in\Z}\log |tjq\beta-M|,
		\]
		and by Lemma~\ref{lem:diophantine} with \(N=tjq\) we have
		\[
		\min_{M\in\Z}\log |tjq\beta-M|
		\gg_E \textcolor{black}{-1}-
		\widehat{h}(tjqP).
		\]
		Putting the last three displays together 
		 gives us
		\[
		\widehat{h}(tjqP)\textcolor{black}{+1\gg_E} \log(y/2t).
		\]
		Recalling that \(\widehat{h}\) is a quadratic form on $E(\Q) \otimes \R$,  we have \(	\widehat{h}(tjqP)=	(tjq)^2 \widehat{h}(P) \) and so
		\[
		\textcolor{black}{-1}+\log (y/t)\ll_E	(tjq)^2 \widehat{h}(P).
		\]
		Provided \(y\geq t^{3/2}\), this implies that \(\log y \ll_E (tjq)^2 \widehat{h}(P) \), which is \eqref{eqn:bound-on-q}.
		
		We substitute  \eqref{eqn:bound-on-q} into \eqref{eqn:exp-sum-primes-majorant} and recall that \( \log\log\log y  \geq 1\) by assumption, to show that either
		\[
		\sum_{\ell\leq y} (\log \ell) e(\ell(m/t+j\beta/2) )
		\ll_{E}
		y(\log y)^{-1}
		+ytj \sqrt{\frac{\widehat{h}(P)}{\log y}} \log\log\log y
		\quad\text{or}\quad y \leq t^{3/2}.
		\]
		In the latter case we have \(\sum_{\ell\leq y} (\log \ell) e(\ell(m/t+j\beta/2))\ll y\leq y^{1/3}t\) by the Prime Number Theorem. So in either case
		\[
		\sum_{\ell\leq y} (\log \ell) e(\ell(m/t+j\beta/2) )
		\ll_{E}
	y^{1/3}t+
		y(\log y)^{-1}
		+ytj \sqrt{\frac{\widehat{h}(P)}{\log y}} \log\log\log y.
		\]
		Since \(\widehat{h}(P)\gg_E 1\) this implies \eqref{eqn:exp-sum-primes}.
	\end{proof}
	
	To apply the previous lemma we turn to the
	Erd\H os--Tur\' an inequality~\cite[Thm.~III]{et49}:
	\begin{lemma}[Erd\H os--Tur\' an]\label{lem:erdos-turan}
		
		For any $0 \leq a < b \leq 1$, any real sequence \(t_m\), any \(M\in\N\) and any \(H>0\) we have
		\[
		\left\lvert
		(b-a)M - 
		\sum_{m=1}^M
		\mathbf{1}_{\{t_m\}\in[a,b)}
		\right\rvert
		\ll
		\frac{M}{H}
		+
		\sum_{1\leq j \leq H}
		\frac{1}{j}
		\left\lvert 
		\sum_{m=1}^M
		e(jt_m)
		\right\rvert,
		\]	
		where we  write \(\{\,\cdot\,\}\) for the fractional part, and \(\mathbf{1}_{\{t_m\} \in[a,b)}=1\) if \({\{t_m\}\in[a,b)}\) and \(0\) otherwise.
	\end{lemma}
	
	We are now ready to prove our  equidistribution result.
	
	\begin{proposition} \label{prop:equidistribution}
		Suppose the point \(P\) from the start of this section satisfies  \(P\in E(\R)^0\).
		Let $s,t \in \N$ with $\gcd(s,t) = 1$ and let $\beta$ be as in Proposition~\ref{prop:silverman-stephens}.
		For any \(0\leq a < b \leq 1\) and any \(\epsilon>0\) we have
		\begin{multline*}
			\#\{ \text{primes } \ell \leq x : \ell \equiv s \bmod t,   \{\ell \beta/2\} \in [a,b) \} 
			=
			\\
			\left(\frac{b-a}{\varphi(t)}+O_{E,\epsilon}
			\left(\frac{t^\epsilon(\log\log\log x)^2{\widehat{h}(P)}}{\log x}\right)^{1/4}
			\right) \frac{x}{\log x},
		\end{multline*}
		where we write \(\{\,\cdot\,\}\) for the fractional part. In particular for any \(\epsilon>0\) {and any \(\sigma\in\{\pm 1\}\)} we have
		\begin{multline*}
			\#\{ \text{primes } \ell \leq x : \ell \equiv s \bmod t, 	(-1)^{\lfloor \ell \beta \rfloor} = {\sigma}\} 
			=
			\\
			\left(\frac{1}{2\varphi(t)}+O_{E,\epsilon}
			\left(\frac{t^\epsilon(\log\log\log x)^2{\widehat{h}(P)}}{\log x}\right)^{1/4}
			\right) \frac{x}{\log x}.
		\end{multline*}
	\end{proposition}
	\begin{proof}
		For the second part, we note that $(-1)^{\lfloor \ell \beta \rfloor}
		=1$ if and only if $0 \leq \{ \ell (\beta/2)\} < 1/2$. So it
		 suffices to prove the first claim in the proposition.
		
		During the proof we will repeatedly use the fact that \(\widehat{h}(P)\gg_E 1\), which holds by for example~\cite[Thm.~VIII.9.10(a)]{Sil09}. If \(t>(\log x)^{1/\epsilon}\) then the bound follows at once from this last result and the Prime Number Theorem. We will assume from now on that \(t\leq (\log x)^{1/\epsilon}\).
		
		Throughout the proof  
		we write $e(y) = \exp(2 \pi i y)$.
		
		We first apply Lemma~\ref{lem:erdos-turan} with \(M,t_m\) as follows. Denote the primes \(\ell\equiv s \bmod t, \ell\leq x\) by \(\ell_1,\dotsc,\ell_M\) and let \(t_m=\{\ell_m\beta/2\}\).
		There is \(c>0\) such that for each \(B>0\) with \(t\leq (\log x)^B\), we have
		\[
		M = \frac{x}{\varphi(t)\log x} +O\left(\frac{x}{\varphi(t)(\log x)^2}\right)+  O_B(x\exp(-c\sqrt{\log x})),
		\]
		by the Siegel-Walfisz theorem~\cite[Corollary~11.21, see also p5]{MV07}. In particular, setting \(B=1/\epsilon\) and recalling that \(t\leq (\log x)^{1/\epsilon}\) and thus \( \exp(c\sqrt{\log x}) \gg_\epsilon t (\log x)^2 \), it follows that
		\[
		{M} = \frac{x}{\varphi(t)\log x} + O\left(\frac{x}{\varphi(t)(\log x)^{2}}\right),
		\]
		We substitute this into 
		Lemma~\ref{lem:erdos-turan} to obtain
		\begin{multline}
			\#\big\{ m\leq M: \{\ell_m\beta/2\}\in[a,b)\big \}-\frac{(b-a)x}{\varphi(t)\log x}
			\\
			\ll
		\frac{x}{H\varphi(t)\log x}
			+\frac{x}{\varphi(t)(\log x)^{2}}
			+
			\sum_{1\leq j \leq H} \frac{1}{j}\left\lvert
			\sum_{\substack{\ell \leq x \\ \ell \equiv s \bmod t}}
			e(j\ell \beta/2)\right\rvert.
			\label{eqn:erdos-turn-applied}
		\end{multline}
		Our goal is to estimate the last sum above. 
		 As often happens it is convenient to count primes weighted by the von Mangoldt function.
		By partial summation,
		\begin{multline}
		\sum_{\substack{\ell \leq x \\ \ell \equiv s \bmod t \\ \ell \text{ prime}}} e(j \ell \beta/2) 
		=
		\frac{1}{\log x}\sum_{\substack{\ell \leq x \\ \ell \equiv s \bmod t\\ \ell \text{ prime}}}
		(\log \ell) e(j \ell \beta/2) \,	+
		\\
		\int_{1}^{x}\frac{1}{y(\log y )^2}
		\sum_{\substack{\ell \leq y \\ \ell \equiv s \bmod t\\ \ell \text{ prime}}}(\log \ell )
		e(j \ell \beta/2)\,\mathrm{d} y.\label{eqn:exp-sum-partial-summation}
		\end{multline}
		It follows from Lemma~\ref{lem:exp-sum} and 	
		 \eqref{eqn:exp-sum-partial-summation} that  for each non-zero integer $j$ we have
		\begin{equation*}
		\sum_{\substack{\ell \leq x \\ \ell \equiv s \bmod t\\ \ell \text{ prime}}}
		e(j \ell \beta/2)
		\ll_E
		\frac{x}{\log x}\cdot tj\sqrt{\frac{\widehat{h}(P)}{\log x}}\log\log\log x.
		\end{equation*}
		Together with  \eqref{eqn:erdos-turn-applied} and the choice
	\[
	H = 
	\left(\frac{\log x}{\widehat{h}(P)}\right)^{1/4}
	\left( t\varphi(t) \log\log\log x\right)^{-1/2},
	\]
	this implies that
	\begin{multline*}
	\#\big\{ m\leq M: \{\ell_m\beta/2\}\in[a,b)\big \}-\frac{(b-a)x}{\varphi(t)\log x}
	\ll
	\\
	\frac{x}{\log x}\cdot\left(\frac{1}{\varphi(t)\log x}
	+
	\left(\frac{\widehat{h}(P)}{\log x}\right)^{1/4}
	\left( \frac{t\log\log\log x}{ \varphi(t) }\right)^{1/2}
	\right).
	\end{multline*}
	The result follows since \(\widehat{h}(P)\gg_E 1\) and \(\varphi(t)\gg_\epsilon t^{1-\epsilon} \).
		\end{proof}

	Our main result is now as follows. In the statement $\ord(\chi(p))$ denotes the multiplicative order of the root of unity $\chi(p)$.
	
	\begin{proposition}\label{prop:odd_valuation}
		Let $\chi$ be a Dirichlet character with modulus $q(\chi)$. Let \(\pi\) be the period of \(\beta_n\bmod q(\chi)\). 
		Suppose that there exists \(\alpha\in\N\) such that
		\begin{equation}
			\gcd(\alpha, \pi )=1
			\label{eqn:gcd_alpha_pi}
		\end{equation}
		and such that one of the following holds:
		\begin{align}
			& \chi(|\beta_\alpha|) \neq  0,1, \quad \text{or} \nonumber \\
			& \chi(-|\beta_\alpha|)\neq 0,1\text{ and } 4\nmid \pi ,
			\quad \text{or} \label{eqn:technical-condition-for-periodicity-3} \\
			& \chi(-|\beta_\alpha|)\neq 0,1 \text{ and }P\in E(\mathbb{R})^{0}. \nonumber
		\end{align}
		Then for any \(\epsilon>0\) and \(x>\exp(\pi^2)\) we have
		\begin{multline*}
			\#\left\{ \text{primes } \ell \leq x :
			\ord(\chi(p)) \nmid v_p( \beta_\ell)
			\text{ for some prime }p \nmid q(\chi) 
			\right \}
			\\
			\geq
			\left(\frac{1}{2\varphi(\pi)}+O_{E,\epsilon}
			\left(\frac{\pi^\epsilon(\log\log\log x)^2{\widehat{h}(P)}}{\log x}\right)^{1/4}
			\right) \frac{x}{\log x}
		\end{multline*}
	\end{proposition}

	\begin{proof}
		From \eqref{eqn:technical-condition-for-periodicity-3}, there is  \(\tau\in\{\pm 1\}\) such that \(\chi(\tau|\beta_\alpha|)\neq 0,1\).
		We separate into two cases depending on the real properties of $P$.
		\subsection*{Case 1. $P\in E(\mathbb{R})^{0}$:}
		From Proposition~\ref{prop:silverman-stephens} we have $\sign(\beta_n) = \sigma^{n-1}(-1)^{\lfloor n \beta \rfloor}$ for some \(\sigma\in\{\pm1\}\) and some irrational number $\beta$. Now consider the set of primes 
		\[\Lambda=\{\ell \text{ prime} : \ell \equiv \alpha \bmod \pi,\,\sign(\beta_\ell) =\tau \sign(\beta_\alpha) \}\}.\]
		Let \(\ell\in\Lambda\). Then by periodicity we have $\beta_\ell \equiv \beta_\alpha \bmod q(\chi)$,  so $\chi(\beta_\ell) = \chi(\beta_\alpha)$ as $\chi$ is periodic modulo $q(\chi)$.  Moreover, we have arranged signs so that \(\chi(|\beta_\ell|)=\chi(\tau|\beta_\alpha|)\neq 0,1\). Hence as $\chi$ is multiplicative we deduce the existence of a prime factor $p$ of $|\beta_\ell|$ with \(p\nmid q(\chi)\) and  $\ord(\chi(p)) \nmid v_p( \beta_\ell).$ It thus suffices to note that \(\{\ell\in\Lambda:\ell\leq x\}\) satisfies the required lower bound by \eqref{eqn:gcd_alpha_pi} and Proposition~\ref{prop:equidistribution}.
		
		\subsection*{Case 2. $P\not\in E(\mathbb{R})^{0}$:} 
		In order to  handle a number of sub-cases simultaneously, we show that there is \(\iota\in\{0,1,2,3\}\) such that \(\alpha+\iota \pi\) is odd and 
		\begin{equation}\label{eqn:sign}
			(-1)^{(\alpha+\iota \pi-1)/2}=
			\begin{cases}
				\tau \sign (\beta_\alpha),& \text{if }\alpha \text{ is even}, \\
				\tau (-1)^{(\alpha-1)/2},& \text{if }\alpha \text{ is odd}.		
			\end{cases}
		\end{equation}

		\subsection*{Case 2.1. \(2 \mid \alpha\)} Here \(\pi\) is odd by 	\eqref{eqn:gcd_alpha_pi}. Choosing \(\iota \in \{1,3\}\) we can arrange for \(\frac{\alpha+\iota\pi-1}{2}\) to be odd or even, and hence \((-1)^{(\alpha+\iota \pi-1)/2}= -1\) or 1 to satisfy  \eqref{eqn:sign}.
		
		\subsection*{Case 2.2. \(2\nmid \alpha\) and \(4\mid \pi\)}  Here we have \(\tau=1\) by \eqref{eqn:technical-condition-for-periodicity-3}. Let \(\iota=0\) and then \((-1)^{(\alpha-1)/2}=\tau(-1)^{(\alpha-1)/2}\) as required for \eqref{eqn:sign}.
		
		\subsection*{Case 2.3.  \(2\nmid \alpha\) and \(4\nmid \pi\)} We can choose \(\iota \in \{0,2\}\) so that \(\iota \pi/2 \) is odd or even  as needed. So we arrange \((-1)^{\iota \pi/2}=\tau\) which gives \eqref{eqn:sign}. \\

		We now let \(q=\lcm(4,\pi)\) and consider primes $\ell$ of the form $\ell\equiv \alpha+\iota\pi \bmod q$. By Proposition~\ref{prop:silverman-stephens} and \eqref{eqn:sign} we then have $\sign(\beta_\ell) = \tau\sign(\beta_\alpha)$. But $\beta_\ell \equiv \beta_\alpha \bmod q(\chi)$ by periodicity, so $\chi(|\beta_\ell|)=\chi(\tau|\beta_\alpha|)\neq 0,1$ by \eqref{eqn:technical-condition-for-periodicity-3} as $\chi$ is periodic modulo $q(\chi)$. We are now in a similar situation to Case 1. Here \eqref{eqn:gcd_alpha_pi} and the fact that $\alpha + \iota \pi$ is odd implies that $\gcd(\alpha+\iota\pi, q) = 1$. Together with the assumption that \(\pi< \sqrt{\log x}\) in the proposition, this allows us to apply the Siegel-Walfisz Theorem~\cite[Corollary~11.21]{MV07} to show that the set under consideration has  size
		\[
		\frac{x}{\varphi(q) \log x}+O\left(\frac{qx}{(\log x)^2}\right)\geq \left(\frac{1}{2\varphi(\pi)} + O\left(\frac{\pi}{\log x}\right)\right)\frac{x}{\log x}.
		\]
		Using again the fact that \(\pi< \sqrt{\log x}\), the claim follows.
		%
		%
		%
	\end{proof}

	\begin{remark} \label{rem:odd_valuation}
		Aside from finitely many exceptions,
		we expect that the primes $p$ constructed in Proposition~\ref{prop:odd_valuation}
		satisfy the stronger condition  $v_p( \beta_{\ell}) = 1$. This is the condition referred to in \S\ref{sec:proof} as being a ``non-Wieferich prime for base $P\in E$''.
	\end{remark}

	\section{Brauer groups} \label{sec:denom}
	The aim of this section is  to prove Theorem~\ref{thm:Brauer_intro}. We begin with some preliminaries on Brauer groups.
	\subsection{Recap of Brauer groups}
	For a scheme $X$ we denote by $\Br X = \rmH^2(X,\Gm)$ its (cohomological) Brauer group. If $X$ is regular and integral and $D \subset X$ is a regular integral divisor, then there is an associated residue map
	$$\res_D : \Br (X \setminus D) [\ell^\infty] \to \rmH^1(D, \Q/\Z)$$
	where $\ell$ is any prime which is invertible on $X$. We say that $b \in \Br (X \setminus D)$ whose order is invertible on $X$ is \emph{unramified} at $D$ if $\res_D(b) = 0$; in which case Grothendieck's purity theorem \cite[Thm.~3.7.1]{CTS20} implies that $b \in \Br X$.
	
	For any $b \in \Br X$ and any point $x \in X$, there is a well-defined specialisation $b(x) \in \Br \kappa(x)$. For a field $k$ we denote by
	\begin{equation} \label{def:X(k)_b}
		X(k)_b = \{ x \in X(k) \colon b(x) = 0 \in \Br k\}
	\end{equation}
	the zero locus of $b$ on $k$-rational points. If only $ b \in \Br \kappa(X)$, then we abuse notation and write 
	$$X(k)_b = \{ x \in X(k) : b \text{ defined at } x, b(x) = 0 \in \Br k\}.$$
	If $\dim X = 1$, then this just means we implicitly remove the finitely many points where $b$ is ramified.
	For a number field $k$, there is an exact sequence
	\begin{equation} \label{eqn:CFT}
		0 \to \Br k \to \bigoplus_v \Br k_v \to \Q/\Z \to 0
	\end{equation}
	where the last map is the sum of all local invariants
	$\inv_v: \Br k_v \to \Q/\Z$ of $k$ \cite[Thm.~13.1.8]{CTS20}. The local invariant is defined in terms of residues, by applying the residue to $\Br \O_v$ then using $\rmH^1(\F_v,\Q/\Z) = \Hom(\Gal(\bar{\F}_v/\F_v),\Q/\Z)$ and evaluating the resulting homomorphism at the Frobenius element \cite[Def.~13.1.7]{CTS20}.
	
	\subsection{Specialisation of Brauer groups on elliptic curves}
	We now prepare for the proof of Theorem~\ref{thm:Brauer_intro}. Let $E$ be an elliptic curve over $\Q$ given by a Weierstrass equation with coefficients in $\Z$. Let $b \in \Br \Q(E)$ which we assume is ramified at some rational point $P$. The residue of $b$ at $P$ is an element of $\rmH^1(\Q, \Q/\Z)$. We associate to this a Dirichlet character $\chi$ as follows.

	\subsubsection{Associated Dirichlet character}\label{sec:assoc-char}
	Firstly the residue of $b$ at $P$ yields a group homomorphism via the identification $\rmH^1(\Q,\Q/\Z) = \Hom(\Gal(\bar{\Q}/\Q), \Q/\Z)$. Let $K/\Q$ be the cyclic extension determined by the kernel. By the Kronecker--Weber Theorem, there is an embedding $K \subseteq \Q(\mu_q)$ where $q$ is the conductor of $K$. As $\Gal(\Q(\mu_q)/\Q) \cong (\Z/q\Z)^\times$ canonically, composing with $\Gal(\Q(\mu_q)/\Q) \to \Gal(K/\Q)$ yields a homomorphism $(\Z/q\Z)^\times \to \Q/\Z$. Recalling that we chose an embedding  $\Q/\Z \subset \C^\times$ in \S\ref{sec:not}, we obtain a primitive Dirichlet character $\chi$ modulo $q$ on extending to $\Z$. (In our work we will only care about the order of $\chi(p)$, which is independent of the choice of embedding $\Q/\Z \subset \C^\times$.) For a prime $p$, this satisfies
	$$\chi(p) = 0 \iff p \mid q, \quad \chi(p) = 1 \iff p \text{ is completely split in }K.$$
	For example, if $K = \Q(\sqrt{D})$ where $D$ is a fundamental discriminant, then we obtain the quadratic character $m \mapsto \left(\frac{D}{m}\right)$ given by the Kronecker symbol.
	
	\subsubsection{Specialisation}
	Let $\mathcal{E}$ be the natural projective model for $E$ over $\Z$ determined by the Weierestrass equation. There exists a non-empty regular open subscheme $\mathcal{E}^\circ \subseteq \mathcal{E}$ such that $b \in \Br \mathcal{E}^{\circ}$.
	
	We choose a finite set of primes $S$ of $\Q$ containing all primes dividing $\ord(\chi)q(\chi)$, where $q(\chi)$ is the conductor of $\chi$.
	We then have the following criterion for triviality of the specialisation, which gives a  direct way of evaluating Brauer group elements via Dirichlet characters.
	
	\begin{proposition} \label{prop:has_Q_p_point}
		Assume that $b$ is ramified at $O$ with associated Dirichlet character $\chi$.
		Let $p \notin S$ and $Q \in \mathcal{E}^\circ(\Q_p) $ with $v_p(Q) \geq 1$. 
		Then $b(Q) = 0 \in \Br \Q_p$ if and only	if $\ord(\chi(p)) \mid v_p(Q)$.
	\end{proposition}
	\begin{proof}
		We let $\mathcal{O}$ and $\mathcal{Q}$
		be the closure of $O$ and $Q$ in $\mathcal{E}_p:=\mathcal{E} \otimes \Z_p$, respectively.
		Note that $\mathcal{O}$ is a smooth subscheme of $\mathcal{E}_p$, 
		since the partial derivative
		with respect to $z$ is non-zero at $O$ modulo $p$.
		As $v_p(Q) \geq 1$, by definition $Q \in E_1(\Q_p)$, so $\mathcal{Q} \equiv \O \bmod p$.
		Thus
		$\mathcal{Q} \cap \mathcal{O} = v_p(Q) O_p$ as a divisor on $\mathcal{Q}$,
		where $O_p = (0:1:0) \bmod p$. 
		
		We now apply \cite[Thm.~3.7.5]{CTS20} with 
		$X = (\mathcal{E}^\circ \otimes \Z_p) \cup \mathcal{O},
		Y = \mathcal{O}$ and
		$f: \mathcal{Q} \to X$ the natural inclusion. This gives that
		the residue $\res_{O_p}(f^* b) \in \rmH^1(O_p,\Q/\Z)$
		is equal to the image of $v_p(Q) \cdot( \res_{\mathcal{O}} b)$
		under the map
		$$\rmH^1(\mathcal{O}, \Q/\Z) \to \rmH^1(O_p,\Q/\Z).$$
		Let $\psi \in \Hom(\pi_1(\Spec \Z_p), \Q/\Z)$ be the homorphism
		corresponding to $\res_{\mathcal{O}} b$. The local invariant $\inv_p (f^*b)$
		is the evaluation of $\res_{O_p}(f^* b)$ at the Frobenius element.
		We thus obtain
		$$\inv_p (f^*b) = v_p(Q) \psi(\Frob _p).$$
		But the Dirichlet character $\chi$ is defined by $\chi(p) = \psi(\Frob_p)$
		after using our choice of embedding $\Q/\Z \subset \C^\times$.
		This gives
		\[\inv_p (f^*b) = 0 \quad \iff \quad \chi(p)^{v_p(Q)} = 1. \qedhere\]
	\end{proof}
	
	\subsection{Proof of Theorem~\ref{thm:Brauer_intro}}\label{sec:brauer_proof}
	
	We now prove our main result. We first note that by Remark~\ref{rem:uniform-period}, the period \(\pi\) is bounded in terms of \(E\) and \(q(\chi)\) only. So we focus on the main bound in the theorem.
	
	We claim that we may assume that $b$ is actually ramified at $O$, rather than just some multiple $mP$ of $P$. Indeed, let $b$ be ramified at $mP$ and consider $t_{-mP}^*b$ where $t_{-mP}$ denotes translation by $-mP$. Then $t_{-mP}^*b$ is ramified at $O$. Assume we have proved the theorem in this case. Then for $n \in \Z$ we have $t_{-mP}^*b(nP) = 0$ if and only if $b(nP + mP) = 0$, so the two sets being counted differ by translation by $m$. This changes the size of the set by at most \( m\).  In conclusion, we may assume that \(m=0\) and $b$ is ramified at $O$. We will do so until the last step of the proof, when we will verify that an error term of size \( m\) can be absorbed into the constant \(C_{E,P,b}\).
	
	We now begin the proof in earnest. Choose a finite set of primes $S$ which contains all primes dividing $\ord(\chi) q(\chi)$ and all primes \(p\) at which \(E\) has bad reduction. Let
	\begin{align}
		\mathcal{P} &= 
		\{\text{primes } p \notin S : \ord(P \bmod p) \text{ is prime and } \ord(\chi(p)) \nmid v_p(\ord(P \bmod p)P)\}, 
		\nonumber
		\\
		\Lambda &= \{ \ord(P \bmod p) : p \in \mathcal{P}\},
		\nonumber
		\\
		T&=
		\{p \text{ prime}: v_p(P)>0\}
		.
		\label{eqn:T}
	\end{align}
	For $p \in \mathcal{P}$, we denote by $\ell_p = \ord(P \bmod p)$.
	To prove the result, we require the following, which comes from our analysis of the elliptic divisibility sequence $\beta_n$ associated to $P$.  We let $\pi$ denote the period of the sequence $\beta_n \bmod q(\chi)$.
	
	\begin{lemma} \label{lem:Lambda}
		For any \(\epsilon>0\) and \(x\gg_{E,\chi} 1\) we have
		\begin{equation*}
			{\#\{ \ell \in \Lambda: \ell \leq x\}}
			\geq
			\left(\frac{1}{2\varphi(\pi)}+O_{\epsilon,E}
			\left(\frac{\pi^\epsilon(\log\log\log x)^2{\widehat{h}(P)}}{\log x}\right)^{1/4}
			\right) \frac{x}{\log x}
			-\#T-\#S.
		\end{equation*}
	\end{lemma}
	One might assume that the \(O_E(\,\cdot\,)\) term is the largest error term here, but at the end of this section it will actually be \(\#T\) which contributes the most to our final bound.

	\begin{proof}
		Consider the multiples $\ell P$ where $\ell$ runs over all primes. Let $p$
		be such that $v_p(\ell P) > v_p(P)$. Then we claim that either \(\ell\in T\), or $\ell = \ord(P \bmod p)$. Indeed,  as $v_p(\ell P) > 0$ we have $\ell P \equiv O \bmod p$. As $\ell$ is prime, we see that either $\ell = \ord(P \bmod p)$ or $P \equiv O \bmod p$. In the latter case we have {\(p\in T\). In that case we also have \(P\in E_{v_p(P)}(\mathbb Q_p), \)  \(P\notin E_{v_p(P)+1}(\mathbb Q_p) \) and \(\ell P \in E_{v_p(P)+1}(\mathbb Q_p)\), and hence \(\ell = |E_{v_p(P)}(\mathbb Q_p) /E_{v_p(P)+1}(\mathbb Q_p) |\)  which is then \(=p\) by Lemma~\ref{lem:formal-group}.  Thus \(\ell \in T\) as claimed.}

		Now Remark \ref{rem:uniform-period}, Proposition~\ref{prop:odd_valuation}, and our assumptions in Theorem~\ref{thm:Brauer_intro}
		imply that there is a set of at least 
		\[
		\left(\frac{1}{2\varphi(\pi)}+O_{E,\epsilon}
		\left(\frac{\pi^\epsilon(\log\log\log x)^2{\widehat{h}(P)}}{\log x}\right)^{1/4}
		\right) \frac{x}{\log x}
		\]
		primes $\ell\leq x$
		such that there exists $p \nmid q(\chi)$ with
		$\ord(\chi(p)) \nmid v_p( \beta_{\ell})$. For such primes we have \(v_p(\beta_\ell)>0\) and hence $v_p(\ell P) > v_p(P)$ by Lemma~\ref{lem:p-adic_beta}. Excluding the finitely many primes \(\ell \in T\), we have $\ell = \ord(P \bmod p)$  by the previous paragraph. 
		So by Lemma~\ref{lem:p-adic_beta} again we have $v_p( \beta_{\ell}) = v_p(\ell P)$. 
		Similarly, as $\ell = \ord(P \bmod p)$, we see that by excluding at most \(\# S\)  of the primes $\ell$ we may assume 
		that $p \notin S$. Such primes now lie in $\Lambda$, hence give the result.
	\end{proof}

	We now sieve modulo such primes.	Our approach is inspired by the version of the elliptic sieve given in \cite[\S 4.4]{kow08} via the large sieve. From a philosophical perspective, we sieve with respect to the maps
	$$E(\Q) \to E(\Z/p^2\Z), \quad p \in \mathcal{P}.$$
	This is literally true providing $v_p(\ell_p P) =1$, but in general we have no control over the size of this valuation, only its value modulo $\ord \chi$ (cf.~Remark~\ref{rem:odd_valuation}). What we actually do is remove suitable multiples of $P$ where we \textit{can} control the valuation. This is a key difference with our approach and that taken in \cite[\S 4.4]{kow08}, as Kowalski only needed to sieve modulo $p$. The precise result is as follows.
	
	\begin{lemma} \label{lem:mod_lp}
		There exists a finite subset $\mathcal{N} \subset \Z$, depending only on \(E\) and \(b\), as follows.
		Let $p \in \mathcal{P}$  and $n \in \Z \setminus \mathcal{N}$
		with $n \equiv \ell_p,2\ell_p,  \ldots, (p-1) \ell_p \bmod p \ell_p$. 
		Then $b(n P) \neq 0 \in \Br \Q_p$.
	\end{lemma}
	\begin{proof}
		As $\gcd(p,n/\ell_p) = 1$, 
		applying Lemma~\ref{lem:formal-group} to $\ell_p P  \in E_1(\Q_p)$ gives
		$$v_p(nP) = v_p(\ell_p P) + v_p((n/\ell_p)) =  v_p(\ell_p P),$$
		which is not divisible by $\ord(\chi(p))$ by the definition of $\mathcal{P}$.
		The result follows from Proposition~\ref{prop:has_Q_p_point} provided we exclude the finitely many points
		in $(\mathcal{E} \setminus \mathcal{E}^\circ)(\Q)$.
	\end{proof}
	
	We now stipulate that the condition in Lemma~\ref{lem:mod_lp} cannot hold at ``moderately sized'' primes $\ell \in \Lambda$, to deduce that the quantity in Theorem~\ref{thm:Brauer_intro} is at most
	$$
	\#\mathcal{N}\cup\big\{ |n| \leq B : (p \in \mathcal{P}\text{ and } \log B \leq \ell_p \leq B)
	\Rightarrow (n \not \equiv 0 \bmod \ell_p \text{ or } n \equiv 0 \bmod p \ell_p)\big\},
	$$
	were \(B\) is a parameter, assumed to be sufficiently large in terms of \(E\) and \(\chi\).
	This is bounded above by $N_0(B) + N_1(B)$ where 
	\begin{align*}
		N_0(B) = & \# \{ |n| \leq B : n \not \equiv 0 \bmod \ell \text{ for all  $\ell \in \Lambda$ with } 
		\log B \leq \ell \leq B \},
		\\
		N_1(B) = & 
		\# \{ |n| \leq B : n \equiv 0 \bmod p \ell_p, \text{ for some $p \in \mathcal{P}$ with } \log B \leq \ell_p \leq B\}.
	\end{align*}
	
	\begin{lemma}\label{lem:Hasse-application}
		$N_1(B) \ll B/\log B$.
	\end{lemma}
	\begin{proof}
		We have
		\begin{align*}
			N_1(B) 
			\ll  \sum_{\substack{\log B \leq \ell_p \leq B}}
			\#\{ |n| \ll B : p \ell_p \mid n \} 
			\ll  \sum_{\substack{\log B \leq \ell_p \leq B}}
			\frac{B}{p \ell_p}. 
		\end{align*}
		However, as $\ell_p$ is the order of $P$ modulo $p$, by the Hasse bounds we have
		$$\ell_p \leq |E(\F_p)| \leq p + 1 + 2\sqrt{p},$$
		in particular $1/p \ll 1/\ell_p$. Extending the sum over all integers $n$ then gives
		\[N_1(B) \ll B \sum_{\substack{\log B \leq \ell \leq B \\ \ell \in \Lambda }}	\frac{1}{\ell^2}
		\ll B \sum_{n \geq \log B}	\frac{1}{n^2} \ll B/\log B.
		\qedhere \]
	\end{proof}
	
	We thus turn our attention to $N_0(B)$, which we  deal with using the Selberg sieve.

	\begin{lemma}\label{lem:Selberg-application}
		$$N_0(B) \leq\exp(\# S+\# T+	O_{E,\epsilon}(\pi^{\epsilon}\widehat{h}(P)^{1/4})
		)
		\frac{B\log\log B}{
			(\log B)^{1/2\varphi(\pi)}}.
		$$
	\end{lemma}
	\begin{proof}
		We  use the version of the Selberg sieve stated in \cite[Thm.~3.6]{MV07}.
		This gives
		$$N_0(B) \ll B \prod_{\substack{ \log B \leq \ell \leq \sqrt{B} \\ \ell \in \Lambda}}
		\left( 1 - \frac{1}{\ell}	\right).$$
		However, by Mertens' theorem we have
		$$\prod_{\substack{ \ell \leq \log B \\ \ell \in \Lambda}}
		\left( 1 - \frac{1}{\ell}	\right)^{-1} \ll 
		\prod_{\substack{ \ell \leq \log B }}
		\left( 1 - \frac{1}{\ell}	\right)^{-1} \ll \log \log B.$$
		Thus it suffices to show that
		$$\prod_{\substack{ \ell \leq \sqrt{B} \\ \ell \in \Lambda}}
		\left( 1 - \frac{1}{\ell} \right) \leq
		\frac{\exp(
			\#S+\#T+O_{E,\epsilon}(\pi^{\epsilon}\widehat{h}(P)^{1/4})
			)}{
			(\log B)^{1/2\varphi(\pi)}}.$$
		To do so, we note that
		\begin{align*}
			\log \prod_{\substack{ \ell \leq \sqrt{B} \\ \ell \in \Lambda}}
			\left( 1 - \frac{1}{\ell} \right)^{-1}	
			& = - \sum_{\substack{ \ell \leq \sqrt{B} \\ \ell \in \Lambda}} 
			\log 	\left( 1 - \frac{1}{\ell} \right)	
			\geq \sum_{\substack{ \ell \leq \sqrt{B} \\ \ell \in \Lambda}} \frac{1}{\ell}.
		\end{align*}
		By partial summation this is
		\[
		\sum_{\substack{ \ell \leq \sqrt{B} \\ \ell \in \Lambda}} \frac{1}{\sqrt{B}}
		+\int_1^{\sqrt{B}}
		\sum_{\substack{ \ell \leq u \\ \ell \in \Lambda}} \frac{1}{u^2}
		\,\mathrm{d}u,
		\]
		and Lemma~\ref{lem:Lambda} shows that this is
		\[
		\geq
		(1/2\varphi(\pi))\log \log B 
		-\# S-
		\#T
		+
		O_{E,\epsilon}(\pi^{\epsilon}\widehat{h}(P)^{1/4}).
		\]
		Exponentiating 
		and taking reciprocals
		gives the claim, and hence the result.
	\end{proof}
	
	Combining these lemmas, and the fact that \(\pi = O_{E,\chi}(1)\) as observed at the start of Section~\ref{sec:brauer_proof}, completes the proof of the bound in Theorem~\ref{thm:Brauer_intro} with
	\begin{equation}\label{eqn:constant}
		C_{E,P,b}
		=
		\exp(\# T+	O_{E,b}(\widehat{h}(P)^{1/4})),
	\end{equation}
	where \(T\) is as in \eqref{eqn:T}. At the start of the proof we assumed that \(m=0\), at the cost of an additive factor of size at most \(m = O_b(1)\). This can  be absorbed into the implicit constant in \eqref{eqn:constant}, and completes the proof.
	\qed
	
	\medskip
	In order to get the last bound in Theorem~\ref{thm:non-abelian}, which depends explicitly on \(P\), we require the following supplement to Theorem~\ref{thm:Brauer_intro}.
	\begin{lemma}\label{lem:explicit-constant}
		Under the assumptions of Theorem~\ref{thm:Brauer_intro}, the constant from \eqref{eqn:constant} satisfies
		\[
		C_{E,P,b}
		= \exp\left(O_{E,b}\left(\frac{\widehat{h}(P)}{\log \widehat{h}(P)}\right)\right).
		\]
	\end{lemma}

\begin{proof}
		By \eqref{eqn:constant} it suffices to show that
		\[
		\#T
		\ll_E
		\frac{\widehat{h}(P)}{\log \widehat{h}(P)}.
		\]
		We have \(\#T=  \omega(\prod_{v_p(P)>0} p) = \omega(e_1)\) where  \(e_1\) is as in Definition~\ref{def:EDSB} and we write \(\omega(k)\) for the number of distinct prime factors of \(k\).
	
	From Definition~\ref{def:EDSB} we have
		\(e_1 \leq \sqrt{H(x(P))}\) where \(H\) is the naive height on \(\mathbb P^1\). Now \(\widehat{h}(P)= O_E(1)+\frac{1}{2}\log H(x(P))\) by~\cite[Thm.~VIII.9.3(e)]{Sil09}. So \(\log e_1\leq O_E(1)+\widehat{h}(P)\), and since  \(\#T=   \omega(e_1)\ll\frac{\log e_1}{\log\log e_1}\) and 
		\(\widehat{h}(P)\gg_E 1\), by say~\cite[Thm.~VIII.9.10(a)]{Sil09}, this completes the proof. 
\end{proof}

%
	\section{Examples and applications} \label{sec:examples}
	
	In this section we give various examples and applications of Theorem~\ref{thm:Brauer_intro}, including the proofs of the results from the introduction and a generalisation of Example~\ref{ex:counter-example}.

	\subsection{A worked example} \label{sec:nice}
	
	Let 
	\begin{equation} \label{eqn:37}
		E:	\quad y^2 +y = x^3 - x.
	\end{equation}
	This  curve has conductor $37$. Its Mordell-Weil group is $\Z$ with generator $P=(0,0)$, which has everywhere good reduction. In particular $M=1$, and Ward's definition (EDSA) agrees with Verzobio's definition (EDSB). The elliptic divisibility sequence associated to $P$ starting at $\beta_0$ reads
	$$0,1, 1, −1, 1, 2, −1, −3, −5, 7, −4, −23, 29, 59, 129, −314, − 65, 1529, −3689,  \dots.$$
	However $P \notin E(\R)^0$ so Theorem~\ref{thm:y} does not apply. Still, we are able to show using Theorem~\ref{thm:Brauer_intro} that the conclusion of Theorem~\ref{thm:y}  holds.

	We apply Theorem~\ref{thm:Brauer_intro} to the quaternion algebra $(-1,y)$. The associated Dirichlet character is then just the non-principal Dirichlet character $\chi$ modulo $4$. Modulo $4$, the EDS starting at $\beta_0$ becomes
	$$0,1, 1, 3, 1, 2, 3, 1, 3, 3, 0, 1, 1, 3, 1, 2, 3, 1, 3, 3, 0, \ldots$$
	which is periodic with period $\pi=10$. One now searches the sequence for terms $\beta_\alpha$ with $\gcd(\alpha,10) = 1$ and $\chi(|\beta_\alpha|) = -1$; one finds that $\beta_7 = -3$ suffices. Thus Theorem~\ref{thm:Brauer_intro} shows that 	$$\#\{n \in \Z  \colon |n| \leq B, y(nP) \text{ is a sum of two squares} \} \ll_\epsilon B/(\log B)^{1/10-\epsilon}$$
	using $1/2\varphi(10) = 1/10$. Alternatively, since $4 \nmid \pi$, we also obtain the result using simply $\beta_1 = 1$ and the fact that $\chi(-|\beta_1|) = 1$.
	
	\begin{remark}
		The keen reader may notice that we did not fully justify our calculation that the period 
		equals $10$. Thankfully, this is not necessary. Namely, Lemma~\ref{lem:period_minimal}
		shows that the period divides
		$$2\cdot (2-1) \cdot 2^2 \cdot  \ord(P \bmod 2) = 2^3 \cdot 5,$$
		whence $\gcd(\pi,7) = 1$. (Note that our bound is correct up to powers of $2$ here.)
	\end{remark}

	We take an example concerning higher order Dirichlet characters.
	Consider again \eqref{eqn:37} and let $\chi$ be a Dirichlet character modulo $7$ of order $3$. One finds that $\chi(|\beta_5|) = \chi(2)$ is non-trivial and that $|E(\F_7)| = 9$. But $\gcd(5, 2\cdot (7-1)|E(\F_7)|) = 1$. Thus by Lemma~\ref{lem:period_minimal} we may apply Theorem~\ref{thm:Brauer_intro}, without even having to calculate the period directly (in fact one finds that the period is $54 = (7-1)|E(\F_7)|$, so our criterion is again best possible up to powers of $2$). The corresponding cyclic extension $K/\Q$ has polynomial $x^3 - x^2 - 2x + 1$; note that this is totally real unlike the hypotheses in Theorem~\ref{thm:intro}. We now apply Theorem~\ref{thm:Brauer_intro} to the cyclic algebra $b=(\chi,x) \in \Br \Q(E)$, which is easily checked to ramify at $O$ with Dirichlet character $\chi^{-2} = \chi$ (since $\ord_O(x) = -2$). We deduce  that 
	$$\#\{n \in \Z  \colon |n| \leq B, x(nP) \text{ is a norm from } K \} \ll B/(\log B)^{\omega}$$
	for some $\omega > 0$.

	As we have seen, this $E$ does indeed satisfy the conclusion of Theorem~\ref{thm:y}. We have counter-examples to similar looking statements (Example~\ref{ex:counter-example}), but it does not seem to be possible to bootstrap these to get counter-examples involving the $y$-coordinate. In particular, we do not know the answer to the following question without imposing additional assumptions on $E$ or the associated EDSB.
	
	\begin{question}
		Let $E$ be an elliptic curve over $\Q$ given by an integral Weierstrass equation.
		Let $P \in E(\Q)$ have infinite order. Then is
		$$\#\{n \in \Z  \colon |n| \leq B, y(nP) \text{ is a sum of two squares} \} = o(B)?$$
	\end{question}

	\subsection{Proof of Theorem~\ref{thm:intro}}
	We translate a statement about conic bundles into a statement about quaternion algebras.
	We work over the local ring $R$ at $Q : = mP$. Restricting the conic bundle to $\Spec R$ we obtain a conic over $R$. It is a classical fact that any conic over $R$ can be diagonalised \cite[Cor.~I.3.4]{HM73}, so we may write the equation for our conic bundle near $Q$ as
	$$sx^2 + ty^2 = z^2, \quad \text{ where } s,t \in R.$$
	However, as we assumed that $X$ is non-singular it is checked that the valuation of $st$ is at most $1$. Thus we may assume without loss of generality that $t$ is a uniformiser in $R$ and that $s$ is a unit in $R$. Let $D$ be the image of $s$ in $R/(t) = \Q$. Since the fibre over $Q$ was assumed to be non-split with imaginary quadratic splitting field, it follows that $D$ is negative. Moreover, up to a suitable change of variables, we may assume that $D$ is a fundamental discriminant. 
	
	We now consider the quaternion algebra $b = (s,t)$ over $\Q(E)$, which gives rise to a $2$-torsion element of $\Br \Q(E)$. The residue of $b$ at $Q$ is $D \in \Q^\times/\Q^{\times 2}$, since $t$ is a uniformiser at $Q$ and $s$ is a unit. It follows that the associated Dirichlet character is the Kronecker symbol $\chi_D(\cdot) = \left(\frac{D}{\cdot}\right)$. As $D$ is negative we have $\chi_D(-1) = -1$, in particular $\chi_D(-|\beta_1|) = -1$. As $P \in E(\R)^0$, it now follows from Theorem~\ref{thm:Brauer_intro} that
	$$\#\{|n| \leq B : b(nP) = 0 \in \Br \Q\} \ll B/(\log B)^{\omega}$$
	for some $\omega > 0$. However, it is clear by construction that  for all but finitely many $R \in E(\Q)$, we have $b(R) = 0 \in \Br \Q$ if and only if $\pi^{-1}(R)$ has a rational point, and the result follows. \qed

	\subsection{Proof of Theorem~\ref{thm:non-abelian}}
	
	Let $L \subset K$ be a cyclic non-totally real subfield. If $y$ is a norm from $K$, then it is certainly a norm from $L$. Thus it suffices to prove the result when $K$ itself is a cyclic non-totally real extension of $\Q$. Let $\chi$ be a Dirichlet character corresponding to $K$ via the Kronecker--Weber Theorem, and consider the cyclic algebra $b = (y,\chi)$. For $Q \in E(\Q)$ with $y(Q) \neq 0$, we have $b(Q) = 0$ if and only if $y(Q)$ is a norm from $K$. The rational function $y$ has a pole of order $3$ at $O$, so it follows that the residue of $b$ at $O$ has Dirichlet character $\chi^{-3}$. However $\chi$ is odd as $K$ has a complex embedding, so $\chi^{-3}$ is also odd. As $\chi(-|\beta_1|) = -1$ and $P \in E(\R)^0$, the main bound in the result thus follows from Theorem~\ref{thm:Brauer_intro} with \(m=0\). The final estimate for the implicit constant follows from Lemma~\ref{lem:explicit-constant}.
	\qed
	
	\subsection{Proof of Theorem~\ref{thm:y}}
	Follows from Theorem~\ref{thm:non-abelian}  for $K = \Q(i)$. \qed
	
	\subsection{$x$--coordinate as sum of two squares}
	
	As the example \eqref{eqn:bad} shows, there are elliptic curves such that $x(Q)$ is a sum of two squares for every $Q \in E(\Q)$. We are able to obtain upper bounds for this counting problem providing one imposes additional assumptions on $E$ and $P$.
	
	\begin{theorem} \label{thm:x}
		Let $E$ be an elliptic curve over $\Q$ given by an integral Weierstrass equation. 
		Let $P \in E(\Q)$ have infinite order with $P \in E(\R)^0$. Assume that
		$x(mP) = 0$ for some non-zero $m \in \Z$.
		Then there exists $\omega=\omega(E,P) > 0$ such that 
		$$\#\{|n| \leq B  \colon x(nP) \text{ is a sum of two squares} \} \ll B/(\log B)^{\omega}.$$
	\end{theorem}
	\begin{proof}
		The relevant conic bundle is given by
		$$x_1^2 + x_2^2 = x x_0^2 \quad \subset \P^2 \times E.$$
		The fibre over $mP$ is non-split since $x(mP) = 0$. Thus the result follows from Theorem~\ref{thm:intro}.
	\end{proof}
	
	This result applies for example to elliptic curves of the form 
	$$y^2 = x^3 + ax^2 + bx + c$$
	where $c$ is a square and one takes $P = (0, \sqrt{c})$, providing $P$ has infinite order and lies in $E(\R)^0$.

	\subsection{Proof of Theorem~\ref{thm:rank_1}}
	Let   \(G=E(\Q)^{\tors} \cap E(\R)^0\). 
	For the rest of this proof, let \(S\) be the set of primes dividing \(\#G\). Choose \(P_0\in E(\Q)\), depending only on \(E\), such that \(E(\Q)\cap E(\R)^0= \langle P_0\rangle \oplus G \)
	where $G$ is finite.
	Then
	\begin{equation}\label{eqn:torsion-dissection}
		\{Q\in E(\Q)\cap E(\R)^0: \widehat{h}(Q)\leq H\}
		=
		\bigcup_{\substack{
				k\in\N, \, k\ll H^{1/2} \\ p\mid k\implies p\in S
		}} \{Q\in \langle kP_0+ G\rangle: \widehat{h}(Q)\leq H\}.
	\end{equation}
	Note that we can discard the case \(\#G=1\), since there the result follows from Theorem~\ref{thm:y}. Then by Mazur's theorem \(\# S = 1\) or 2.
	
	We now count points \(Q\in  E(\Q)\cap E(\R)^0\) such that \(y(Q)\) is a sum of two squares. 
	To do this, we break the union in \eqref{eqn:torsion-dissection} up into two parts according to the size of \(k\). First, for small \(k\) we will apply Theorem~\ref{thm:non-abelian}, and then for large \(k\) we will use a trivial bound. 
	
	The case \(K=\Q(i)\) of Theorem~\ref{thm:non-abelian} implies that for each \(P\in E(\Q)\cap E(\R)^0\) of infinite order, we have
	\begin{equation*}\label{eqn:thm-y-for-application}
		\#\left\{ Q\in \langle P\rangle :
		\begin{array}{l}
				\widehat{h}(Q) \leq H, y(Q) \text{ is a }\\
				\text{sum of two squares} 
		\end{array} 
		\right\}
		\leq \exp\left(O_{E}\left( \frac{\widehat{h}(P)}{\log \widehat{h}(P)} \right)\right)
		\frac{H^{1/2}}{(\log H)^{\omega}},
	\end{equation*}
	for some $\omega > 0$ which depends only on \(E\).
	Let \(\kappa\) be a positive integer to be chosen later. By \cite[Prop.~VIII.9.6]{Sil09} we have \(\widehat{h}(kP_0+R)= C_{E,P_0} k^2\) for all torsion points $R$ and some constant \(C_{E,P_0}\) depending only on \(E\) and \(P_0\). Since \(P_0\) depends only on \(E\), it follows  from the last display that
	\begin{align}\label{eqn:applying-thm-y}
		\begin{split}
			\sum_{{
					k\in\N, \, k\leq \kappa
			}}
			\# \left\{Q\in \langle kP_0+G\rangle: 
			\begin{array}{l}
				\widehat{h}(Q) \leq H, y(Q) \text{ is a }\\
				\text{sum of two squares} 
			\end{array} \right\}
			\\
			\leq\exp\left(O_{E}\left( \frac{\kappa^2}{\log \kappa} \right)\right)
			\frac{H^{1/2}}{(\log H)^\omega}.
		\end{split}
	\end{align}
	To handle large \(k\) we claim that
	\begin{equation}
	\label{eqn:large-k}
	\sum_{\substack{
			k\in\N \\  \kappa<k\ll H^{1/2}\\p\mid k\implies p \in S
	}}
	\# \{Q\in \langle kP_0+G\rangle: \widehat{h}(Q)\leq H\}
	\ll_E \frac{H^{1/2}(\log \kappa)^{\# S-1}}{\kappa}.
	\end{equation}
	 We will prove this in the case \(\#S=2\), leaving the similar and slightly simpler case \(\# S=1\) to the reader.
	 Recalling that  \(\widehat{h}(kP_0+R)= C_{E,P_0} k^2\) for all \(R\in G\), we have 
	\[\# \{Q\in \langle kP_0+G\rangle: \widehat{h}(Q)\leq H\}\ll_E  H^{1/2}k^{-1}\]
	and so
	\begin{align*}
		\sum_{\substack{
				k\in\N \\  \kappa<k\ll H^{1/2}\\p\mid k\implies p \in S
		}}
		\# \{Q\in \langle kP_0+G\rangle: \widehat{h}(Q)\leq H\}
		 &\ll_{E} 
		\sum_{\substack{
				a,b\in\N\cup\{0\}, \\  
				 \kappa<p_1^a p_2^b\ll H^{1/2}
		}}
	H^{1/2}p_1^{-a}p_2^{-b}
	\\
	&\leq	
	H^{1/2}\sum_{\substack{j\in \N\\ \kappa \leq 2^j\ll H^{1/2}}}
		\sum_{\substack{
				a,b\in\N\cup\{0\}, \\  
				2^{j-1}<p_1^a p_2^b\leq 2^j
		}}\frac{1}{2^{j-1}},
%
%
%
	\end{align*}
	and since there are at most \(j\) pairs \((a,b)\) appearing in the final sum, this is at most \(
	H^{1/2}\sum_{{2^j\geq \kappa}}
	{j}/{2^{j-1}} \), which is \(O(	{H^{1/2}\log \kappa/\kappa})\) as required
	for \eqref{eqn:large-k}.
	
	Combining \eqref{eqn:large-k} with \eqref{eqn:torsion-dissection} and \eqref{eqn:applying-thm-y} gives
	\begin{multline*}
			\#\left\{Q \in E(\Q) \cap E(\R)^0 :
			\begin{array}{l}
				\widehat{h}(Q) \leq H, y(Q) \text{ is a }\\
				\text{sum of two squares} 
			\end{array}
			\right\}
		\\
		\ll_{E}
		\frac{H^{1/2}(\log \kappa)^{\# S-1}}{
			\kappa}
		+
		\exp\left(O_{E}\left( \frac{\kappa^2}{\log \kappa} \right)\right)
		\frac{H^{1/2}}{(\log H)^\omega}.
	\end{multline*}
	We choose \(\kappa= (\epsilon_{E,K}\log \log H \log \log \log H)^{1/2}\) for some small \(\epsilon_{E,K}>0\), so that \(\frac{\kappa^2}{\log \kappa} \ll \epsilon_{E,K} \log \log H \), and we then obtain the claimed result in the form
	\begin{multline*}
			\#\left\{Q \in E(\Q) \cap E(\R)^0 :
			\begin{array}{l}
				\widehat{h}(Q) \leq H, y(Q) \text{ is a }\\
				\text{sum of two squares} 
			\end{array}
			\right\}
		 	 \ll_{E} \frac{H^{1/2}(\log\log \log H)^{\# S-3/2}}{(\log\log H)^{1/2}}.
	\end{multline*}

	\subsection{Unramified elements}

	Our results contain various technical assumptions. In the next two sections we demonstrate that these are necessary in general. Firstly we show that for the conclusion of Theorem~\ref{thm:Brauer_intro}, we need to impose ramification on the Brauer group elements. (See \eqref{def:X(k)_b} for the notation $E(k)_b$.)
	
	\begin{lemma} \label{lem:unramified}
		Let $E$ be an elliptic curve over a number field $k$.
		Let $b \in \Br E$ with $E(k)_b \neq \emptyset$.
		Then $E(k)_b$ contains a translate of a subgroup of finite index.
		In particular $E(k)_b$ has positive density in $E(k)$.
	\end{lemma}
	\begin{proof}
		By performing a translation, we may assume that $O \in E(k)_b$.
		As $b$ is unramified, there exists a finite set of places $S$ such that 
		$E(k_v)_b = E(k_v)$ for every $v\notin S$ \cite[Prop.~13.3.1(iii)]{CTS20}.
		Moreover, for any place $v$ by \emph{loc.~cit.} the evaluation map
		$$b: E(k_v) \to \Br k_v$$
		is locally constant.  
		In particular, there exists an open neighbourhood $O \in U_v$
		such that $U_v \subset E(k_v)_b$. But $E(k_v)$ is profinite, so such an open neighbourhood
		may be refined to an open subgroup $U_v$ of $E(k_v)$, which neccessarily has finite index
		as $E(k_v)$ is compact. So set $A = E(k) \cap_{v \in S} U_v$. By construction,
		this is a subgroup of $E(k)$ of finite index. Moreover, for all $P \in A$
		and all places $v$ we have $b(P) = 0 \in \Br k_v$. It now follows from the Hasse
		principle for $\Br k$ \eqref{eqn:CFT} that $A \subset E(k)_b$, as required.
	\end{proof}
	
	Let 
	$$b \in \Be(E) = \ker(\Br E \to \prod_v \Br E_{k_v})$$
	where the product is over all places $v$ of $k$. It follows easily from \eqref{eqn:CFT} that $b(P) = 0 \in \Br k$ for all $P \in E(k)$, so here $E(k)_b = E(k)$ even if $b \neq 0$. Such elements exactly correspond to the elements of $\Sha(E)$, providing it is finite (see \cite[Thm.~6.2.3]{Sko01}). 
	
	For completeness, we give such an explicit example in the form of a conic bundle. Our example is based on a variant of the well-known counter-example to the Hasse principle $2y^2 = x^4 -17$ due to Reichardt and Lind.
	
	\begin{proposition}
		Let $N$ be only divisible by primes which are $1 \bmod 8$ and consider the elliptic
		curve
		$$E: \quad y^2 = x(x^2 + N).$$
		Let $X$ be a smooth proper model for the the conic bundle over $E$ given by
		$$t_0^2 - 2t_1 = x t_2^2.$$
		\begin{enumerate}
			\item The map $X(\Q) \to E(\Q)$ is surjective.
			\item For $N = 17 \times 593$, $E$ has positive rank and the conic bundle
			morphism admits no section.
		\end{enumerate}
	\end{proposition}
	\begin{proof}
		(1) We use the equation for the curve $y^2 = xz(x^2 + Nz^2)$ in weighted projective
		space. The conic bundle $X$ then has the equations 
		\begin{equation} \label{eqn:Sha_1}
			t_0^2 - 2 t_1^2 = xz t_2^2, \quad \text{if } xz \neq 0
		\end{equation}
		and
		\begin{equation} \label{eqn:Sha_2}
			t_0^2 - 2 t_1^2 = (x^2 + N z^2)t_2^2, \quad \text{if }  x^2 + Nz^2 \neq 0.
		\end{equation}
		We now verify that $X(\Q_v) \to E(\Q_v)$ is surjective for all places $v$ of $\Q$.
		First note that $X(\Q_v) \to E(\Q_v)$ is closed with respect to the $v$-adic topology.
		So it just suffices	to show that there is $\Q_v$-point over each fibre with 
		$xz(x^2 + Nz^2) \neq 0$.
		
		For $v = \infty$, there is clearly the real point $(\sqrt{2} : 1: 0)$ in
		\eqref{eqn:Sha_1}. This also
		gives a solution	for any $p \equiv 1 \bmod 8$, in particular for all $p \mid N$.
		So let $p \nmid 2N$. Here we may choose a representative so that $p$ doesn't simultaneously
		divide $x$ and $z$. If $v_p(x^2 + N z^2)$ is even, then from the equation of $E$ we find that
		$v_p(xz)$ is even, whence \eqref{eqn:Sha_1} has a solution by a Hilbert symbol calculation.
		If $v_p(x^2 + N z^2)$ is odd, our assumptions imply that $v_p(xz) = 0$ so again a Hilbert symbol
		calculation shows that \eqref{eqn:Sha_1} has a solution.

		Finally, for $p = 2$, here $N \equiv 1 \bmod 8$ is a square in $\Q_2^\times$,
		so making a change of variables we may assume that $N = 1$. 
		Again  from the equation of the curve, we may assume that
		$2$ doesn't simultaneously divide $x$ and $z$.
		First suppose that $2 \nmid xz$. Then from the equation $v_2(x^2 + z^2)$
		is even, which is a contradiction as $x^2 + z^2 \equiv 2 \bmod 8$.
		So assume without loss of generality that $2 \mid x$.
		Then $z,x^2 + z^2$ are both odd, hence from the equation $v_2(x)$ is even.
		Then $x^2 + z^2 \equiv 1 \bmod 8$, in which case \eqref{eqn:Sha_2}
		has a solution in $\Q_2$. 
		So every fibre is everywhere locally soluble,
		hence has a rational point. This proves (1).
		
		(2) Here \texttt{Sage} verifies that $E$ has torsion subgroup isomorphic to $\Z/2\Z$.
		But it has the point $(1088:36040:1)$, which must have infinite order (in fact
		\texttt{Sage} verifies by $2$-descent that $\rank E(\Q) = 2$). The conic bundle corresponds
		which by (1) and \eqref{eqn:CFT} can only happen if $b$ is trivial.
		But, by the general theory of $2$-descent for elliptic curves, the element $b$
		corresponds to a $2$-covering of $E$; this $2$-covering has the equation
		$$C_2: 2w^2 = 2^2 - 4Nz^4$$
		(cf.~\cite[Ex.~X.4.8]{Sil09}), which after a change of variables gives
		$C_2: 2w^2 = z^4 - N$.
		However this curve fails the Hasse principle; for $N=17$ this is the well-known example of 
		Reichardt and Lind, whereas in our case $N = 17 \times 593$ this follows from
		a similar argument to Case I of the proof of \cite[Prop.~X.6.5]{Sil09}. (The key point
		is that $2$ is not a quartic residue modulo $17$, see \cite[\S1]{Poo01}).
		It thus follows that this $2$-covering is non-trivial, so $b$ is non-trivial, as required.
	\end{proof}

	\subsection{Ramified examples}
	Lemma~\ref{lem:unramified} shows that to get any hope of obtaining a version of Theorem~\ref{thm:Brauer_intro} there needs to be ramification. Still this assumption is not sufficient in general, as was first shown for conic bundles in \cite[\S3]{bn19}. In this section, we build on this example and generalize the construction to higher order Brauer group elements. We work over a  number field $k$.

	\begin{definition}
		Let $P \in E(k)$ and $n \in \N$.
		We denote by $E_P[n]$ the scheme of points $Q \in E$ with $nQ = P$.
		This is a $E[n] = E_O[n]$ torsor. 
	\end{definition}

	\begin{proposition} \label{prop:ramified}
		Let $\ell$ be a prime, $E$ an elliptic curve over $k$, and $P \in E(k)$
		a primitive point of infinite order.
		Let $b \in \Br k(E)$ be ramified with ramification locus in $\Z P \setminus \ell \Z P$,
		whose residues have the same cyclic extension $K/k$.
		Assume that $K$ is a subfield of the field of fractions of every irreducible
		component of $E_P[\ell]$. 
		
		If $O \in E(k)_b$ then $E(k)_b$ contains a subgroup of finite index.
		In particular $E(k)_b$ has positive density in $E(k)$.
	\end{proposition}
	\begin{proof}
		Choose a finite set of places $S$ containing the archimedean places, the places ramified in $K$, and the places where $E_P[\ell]$ is not finite \'etale.
		
		If $v$ is completely split in $K$, then $b \otimes k_v$ is unramified, hence $E(k_v)_b = E(k_v)$ providing $S$ is sufficiently large \cite[Prop.~7.1]{LM19}. So assume that $v$ is not completely split in $K$. We claim that
		\begin{equation} \label{eqn:Zell + 1}
			\ell E(k) \subset E(k_v)_b.
		\end{equation}
		To see this, assume instead that $\ell R \notin E(k_v)_b$ for some $R \in E(k)$. Then providing $S$ is sufficiently large, by \cite[Prop.~7.1]{LM19} we must have $\ell R \equiv Q \bmod v$ where $Q \in E(k)$ is a ramification point of $b$. Since $Q \in \Z P \setminus \ell \Z P$ by assumption, we deduce that $\ell R \equiv qP \bmod v$ for some $q\in \Z$ with $\ell \nmid q$.  So $qP \bmod v$ is $\ell$-divisible. But $\ell \nmid q$, hence $P \bmod v$ is also $\ell$-divisible. But as $E_P[\ell]$ is finite \'etale over $v$, Hensel's Lemma shows that $E_P[\ell](k_v) \neq \emptyset$. However now $K$ is a subfield of the field of fractions of every irreducible component of $E_P[\ell]$, so we deduce that $K$ admits a place of degree $1$ over $v$. This  contradicts that $v$ is not completely split in $K$ and shows \eqref{eqn:Zell + 1}.
		
		We have shown that $\ell E(k) \subset E(k_v)_b$ for all $v \notin S$. For the places in $S$ we proceed as in the proof of Lemma~\ref{lem:unramified}, using the fact that $O \in E(k)_b$ to deduce that $E(k_v)_b$ contains a subgroup of finite index for all $v$. The result then follows from the fact that the intersection of finitely many finite index subgroups has finite index.
	\end{proof}
	
	We now explain how to construct such explicit examples. Let $\ell$ be a prime and $E/k$ an elliptic curve with primitive point $P$ of infinite order. Consider the natural map $f: E \to \P^1$ given by projecting to the $x$-coordinate. Let $K/k$ be a non-trivial cyclic extension. By the Faddeev exact sequence \cite[Thm.~1.5.2]{CTS20}, there exists $a \in \Br k(\P^1)$ whose ramification consists of any given collection of $2$ distinct rational points and whose residue at these points is $K/k$.
	
	We therefore choose distinct non-Weierstrass points $P_1,P_2 \in (\ell\Z + 1) P$ and choose $a$ ramified at $f(P_1),f(P_2)$. We then define $b' = f^*a$. This is ramified at the points $\pm P_i$ with residues in $K/k$. However we may have $E(k)_{b'} = \emptyset$. So we set $b = b' - b'(O)$, which now satisfies $O \in E(k)_{b}$.
	
	It remains to arrange that $K$  is a subfield of the field of fractions of every irreducible	component of $E_P[\ell]$, for which we need to impose further conditions on $E$ and $K$. We assume $E$  has full $\ell$-torsion, which  implies that $k$ contains all $\ell$th roots of unity. As $E_P[\ell]$ is an $E[\ell]$-torsor, we find that it corresponds to some element of 
	$$\HH^1(k,E[\ell]) = \HH^1(k, \mu_\ell^2) = k^\times/k^{\times \ell} \times k^\times/k^{\times \ell},$$
	where the last isomorphism is by Kummer theory. A pair $(\alpha,\beta) \in k^\times/k^{\times \ell} \times k^\times/k^{\times \ell}$ corresponds to the $\mu_\ell^2$-torsor given by
	\begin{equation} \label{eqn:torsor}
		x^\ell = \alpha, \quad y^\ell = \beta \qquad \subset \A^2_k.
	\end{equation}
	Now, the torsor $E_P[\ell]$ is non-trivial as $P$ is not $\ell$-divisible, so without loss of generality $\alpha$ is not an $\ell$th power in the notation of \eqref{eqn:torsor}. But then the function fields of the irreducible component of the scheme \eqref{eqn:torsor} are either $k(\sqrt[\ell]{\alpha})$ or $k(\sqrt[\ell]{\alpha}, \sqrt[\ell]{\beta})$. These  contain the field $K=k(\sqrt[\ell]{\alpha})$ which is non-trivial and cyclic of degree $\ell$.
	
	Writing down explicit curves and Brauer group elements which satisfy Proposition~\ref{prop:ramified} is now relatively easy using explicit ramified cyclic algebras on $\P^1$. This is how we found Example~\ref{ex:counter-example}.

	\section{Applicability} \label{sec:100}
	
	In this final section we verify that Condition \eqref{itm:nonresidue-I} in Theorem~\ref{thm:Brauer_intro} holds for almost all Dirichlet characters, under suitable assumptions. Before restating the condition, we recall our setup along with the notation: Fix an elliptic curve $E$ over $\Q$ with $P \in E(\Q)$ a point of infinite order with \(\beta_n\)  the associated EDSB. Let $\chi$ be a Dirichlet character with modulus $q(\chi)$, and $\pi(\chi)$ be the period of the sequence $\beta_n\bmod q(\chi)$. 
	
	Let us recall the technical condition. 
	\begin{equation}\label{eqn:*2}
		\textnormal{There exists }\alpha\in\Z\textnormal{ relatively prime to }\pi(\chi)\textnormal{ such that }\chi(|\beta_\alpha|)\notin \{0,1\}.
	\end{equation}
	We expect this to hold for all but finitely many Dirichlet characters, but it seems completely out of reach to prove this at present. Since this section is purely illustrative, we make numerous assumptions to simplify the statements and the proof.
	
	We assume $E$ does not have complex multiplication and that $P=(x,y)$ has integer coordinates with everywhere good reduction. (This is satisfied by the pair $(E,P)$ from \S\ref{sec:nice}, for example.) Under these assumptions, we have
	$$\beta_n=\psi_n(P),\ |\beta_n|=|e_n|.$$
	For odd Dirichlet characters some of our other conditions are more likely to apply, so for simplicity we consider only even Dirichlet characters, which we also assume to have prime modulus:
	\[ \Sigma(D) = \{\text{Dirichlet characters } \chi : \chi(-1)=1, q(\chi) \text{ is prime}, q(\chi) \leq D \}.\]
	
	\begin{theorem}\label{thm:100}
		$$\lim_{D \to \infty} \frac{\#\{ \chi \in \Sigma(D) : \chi \text{ satisfies } \eqref{eqn:*2}\}}{\#\Sigma(D)} = 1,$$
		i.e.~$100\%$ of Dirichlet characters in $\Sigma(D)$ satisfy \eqref{eqn:*2} as $D$ tends to infinity.
	\end{theorem}

	\begin{proof}
		As we only consider even characters we have $\chi(|\beta_\alpha|)=\chi(\beta_\alpha)$. If an $\alpha$ satisfying \eqref{eqn:*2} exists, then considering the arithmetic progression $\alpha \bmod \pi(\chi)$ shows that there exists such an $\alpha$ with $\alpha$ prime.  So let $\Omega$ be the set of all primes $\ell > 3$ with $|\beta_\ell| \neq 1$ (there are easily seen to be only finitely many such primes by Seigel's theorem on integral points on elliptic curves). For any subset $R\subseteq\Omega$, define  the sets
		\begin{align*}
			\Phi(D,R)&\coloneqq\left\{\chi \in \Sigma(D) :\forall\ell\in R \textnormal{ either }\chi(\beta_\ell)\in\{0,1\}\textnormal{ or }\ell\mid \pi(\chi)\right\}, \\
			\Phi'(D,R)&\coloneqq\left\{\chi \in \Sigma(D) : \forall\ell\in R \textnormal{ either }\chi(\beta_\ell)\in\{0,\pm1\}\textnormal{ or }\ell\mid \ord(P \bmod q(\chi))\right \}, \\
			\Phi''(D,R)&\coloneqq\left\{\chi \in \Sigma(D) : \forall\ell\in R \textnormal{ either }\chi(\beta_\ell)\in\{\pm1\}\textnormal{ or }\ell\mid \ord(P \bmod q(\chi))\right\}.
		\end{align*}
		Note that $\Phi(D,\Omega)$ contains all characters in $\Sigma(D)$ that fail \eqref{eqn:*2}. It suffices to show
		\begin{equation} \label{eqn:Phi}
			\lim_{D\to\infty}\frac{\#\Phi(D,\Omega)}{\#\Sigma(D)}=0.
		\end{equation}
		
		\begin{lemma}\label{lem:limitsecond}
			$\Phi(D,\Omega)\subseteq\Phi'(D,\Omega)=\Phi''(D,\Omega)$.
		\end{lemma}
		
		\begin{proof}

			We first prove $\Phi(D,\Omega)\subseteq \Phi'(D,\Omega)$. Suppose $\chi\in\Phi(D,\Omega)$ but $\chi\notin\Phi'(D,\Omega)$. Let $q=q(\chi)$ be the modulus of $\chi$. Then there exists a prime $\ell$ such that $\chi(\beta_\ell)\neq0,\pm1$ and $\ell\nmid \ord(P \bmod q)$, but $\ell\mid\pi(\chi)$. Let $\rho=\pi(\chi)/\ell^{v_\ell(\pi(\chi))}$. Note that $\ord(P \bmod q)\mid\rho$ and so $q\mid \beta_\rho$. For any integer $k$, we have the following implications,
			$$q\mid \beta_{\ell+k\rho}\implies \ord(P \bmod q)\mid \ell+k\rho\implies \ord(P \bmod q)\mid \ell\implies \ord(P \bmod q)=\ell.$$
			However, since $\ell \nmid \ord(P \bmod q)$ it follows $\chi(\beta_{\ell+k\rho})\neq0$ for any $k\in\Z$. Note that $\gcd(\pi(\chi),\ell+\rho)=\gcd(\pi(\chi),\ell-\rho)=1$ by construction of $\rho$. Hence, there exists $k_1,k_2\in\Z$ such that
			$$
			\ell_1=\ell+\rho+k_1\pi(\chi),\quad \ell_2=\ell-\rho+k_2\pi(\chi)
			$$
			are both primes in $\Omega$ and $\ell_1,\ell_2\nmid \pi(\chi)$. Then we must have
			$\chi(\beta_{\ell_1}),\chi(\beta_{\ell_2})=1$ since $\chi \in \Phi(D,\Omega)$. By periodicity, $\chi(\beta_{\ell\pm\rho})=1$ as well. Using Proposition~\ref{prop:Verzobio} (noting that $M = 1$ in our case) with $n=\ell, m = \rho$ and $r = 1$, we have
			\begin{equation}\label{eqn:d_ell}
				\beta_{\ell+\rho}\beta_{\ell-\rho}\equiv\beta_{\rho+1}\beta_{\rho-1}\beta_{\ell}^2\bmod q
			\end{equation}
			Since $\chi(\beta_{\ell+\rho}\beta_{\ell-\rho})=1$ and $\chi(\beta_\ell^2)\neq0,1$, it follows that $\chi(\beta_{\rho+1}\beta_{\rho-1})\neq0,1$. 
			Using Proposition~\ref{prop:Verzobio} again with $n = \ell + 2\rho, m = \rho, r= 1$, we have
			$$\beta_{\ell+3\rho}\beta_{\ell+\rho}\equiv\beta_{\rho+1}\beta_{\rho-1}\beta_{\ell+2\rho}^2\bmod q.$$
			Hence, either $\chi(\beta_{\ell+2\rho})\neq0,1$ or $\chi(\beta_{\ell+3\rho})\neq0,1$. In either case, after choosing $k_1,k_2\in\Z$ such that
			$$
			\ell+2\rho+k_1\pi(\chi),\quad \ell+3\rho+k_2\pi(\chi)
			$$
			are both primes in $\Omega$ (we use the assumption $\ell>3$ here), we obtain a prime $\ell_0\in\Omega$ such that $\chi(\beta_{\ell_0})\neq0,1$ and $\ell_0\nmid \pi(\chi)$, which contradicts $\chi \in \Phi(D,\Omega)$. This finishes the proof that $\Phi(D,\Omega)\subseteq \Phi'(D,\Omega)$.
			
			The containment $\Phi''(D,\Omega)\subseteq \Phi'(D,\Omega)$ is clear. To show $\Phi'(D,\Omega)\subseteq \Phi''(D,\Omega)$, it suffices to prove the following implication
			\begin{equation}\label{eqn:orderimplication}
				\chi(\beta_\ell)=0\implies \ell\mid\ord(P \bmod q)
			\end{equation}
			for all $\ell\in\Omega$. If $\chi(\beta_\ell)=0$, then $\ord(P \bmod q)\mid \ell$. However, since $\ord(P \bmod q)\neq1$, we must have $\ord(P \bmod q)=\ell$. This establishes \eqref{eqn:orderimplication}, and hence $\Phi'(D,\Omega)=\Phi''(D,\Omega)$.
		\end{proof}
		
		\begin{lemma}\label{lem:ladicrep}
			Let $\ell$ be a prime such that the $\ell$-adic Galois representation for $E$ $\rho_\ell\colon\Gal(\bar{\Q}/\Q)\to\GL_2(\Z/\ell\Z)$ is surjective. Then
			$$\limsup_{D\to\infty}\frac{\#\{\chi\in \Sigma(D) : \ell\mid \ord(P \bmod q(\chi))\}}{\#\{\chi\in \Sigma(D)\}}\leq\frac{\ell}{\ell^2-1}.$$
		\end{lemma}
		
		\begin{proof}
			To calculate the limsup, we can ignore finitely many characters. Hence, we will ignore characters whose modulus is ramified in the extension $\Q(E[\ell])$. Let $q$ be a prime unramified in $\Q(E[\ell])$ such that $\ell\mid \ord(P \bmod q)$. This implies that $E(\F_q)[\ell]\neq 0$,
			which is equivalent to
			\begin{equation}\label{eqn:det}
				\det(I_2-\rho_\ell(\Frob_q))=0.
			\end{equation}
			There are $\ell(\ell+1)(\ell-1)^2$ elements in $\GL_2(\Z/\ell\Z)$, and among them there are $\ell^3-\ell^2$ many elements $x\in\GL_2(\Z/\ell\Z)$ that satisfy the equation $\det(I_2-x)=0$.
			Thus, by the Chebotarev density theorem the proportion of primes $q$ that satisfy \eqref{eqn:det} is
			$$\frac{\ell^3-\ell^2}{\ell(\ell+1)(\ell-1)^2}=\frac{\ell}{\ell^2-1}.$$
			Let $\Omega_\ell$ denote the set of primes that satisfy \eqref{eqn:det}. Recall that if $\ell\mid \ord(P \bmod q)$, then $q\in\Omega_\ell$. As there are $q-1$ Dirichlet characters of modulus $q$, the limit in question is less than or equal to
			$$\lim_{D\to\infty}\frac{\#\{\chi\in \Sigma(D) : q(\chi) \in\Omega_\ell\}}{\#\{\chi\in \Sigma(D)\}}=\lim_{D\to\infty}\frac{\displaystyle\sum_{q\in \Omega_\ell,\; q\leq D} q-1}{\displaystyle\sum_{q\leq D} q-1}=\frac{\ell}{\ell^2-1}$$
			where the last equality follows from  simple application of the Chebotarev density theorem and partial summation.
		\end{proof}
		
		We now show \eqref{eqn:Phi}. To do so we consider a finite subset $R\subset \Omega$, then take the limit over all $R$. We divide the quantity of interest into two parts as
		\begin{align*}
			\frac{\#\Phi''(D,R)}{\#\Sigma(D)}\leq & \frac{\#\left\{\chi \in \Sigma(D) :\chi(\beta_\ell) \in \{\pm 1\}\textnormal{ for some }\ell\in R\right\}}{\#\Sigma(D)}\\
			&+\frac{\#\{\chi\in \Sigma(D) : \ell\mid \ord(P \bmod q(\chi))\; \forall\ell\in R\}}{\#\Sigma(D)}.
		\end{align*}
		We start with the first part. Fix some $\ell\in R$ and let $q \leq D$ be any prime not dividing $\beta_\ell$. Let $N$ be the order of $\beta_\ell$ in $(\Z/q\Z)^{\times}/\{\pm1\}$. Then the map
		$$\{\chi\textnormal{ modulus }q : \chi(-1)=1\} \to \mu_{N}, \quad \chi \mapsto \chi(\beta_\ell)$$
		is surjective. Hence 
		\begin{equation}\label{eqn:order}
			\frac{\#\{\chi\textnormal{ modulus }q: \chi(-1)=1,\; \chi(\beta_\ell) \in \{ \pm 1\}\}}{\#\{\chi\textnormal{ modulus }q : \chi(-1)=1\}}=\frac{2}{N}.
		\end{equation}
		Recall that by the definition of $\Omega$ we have $|\beta_\ell| \neq 1$. It follows that if $q > |\beta_\ell|$, then $N\geq \log ((q-1)/2)/\log |\beta_\ell|$. By taking $q$ very large compared to $\ell$ and sorting characters by their modulus, as well as safely ignoring characters of small modulus, one finds that
		$$
		\lim_{D\to\infty}\frac{\#\left\{\chi \in \Sigma(D) :\chi(\beta_\ell) \in \{\pm 1\}\textnormal{ for some }\ell\in R\right\}}{\#\Sigma(D)}=0.$$
		So, we are left with only the second part. Since $E$ does not have complex multiplication, the $\ell$-adic Galois representation $\rho_\ell\colon\Gal(\bar{\Q}/\Q)\to\GL_2(\Z/\ell\Z)$ is surjective outside a finite set of primes by Serre's open image theorem \cite[Thm.~III.7.9]{Sil09}. Hence, shrinking $\Omega$ by finitely many primes if necessary, we can assume that $\rho_\ell$ is surjective for all $\ell\in \Omega$.  
		Thus Lemma~\ref{lem:ladicrep} gives
		$$
		\lim_{D\to\infty}\frac{\#\Phi''(D,R)}{\#\Sigma(D)} \leq \min\left(\frac{\ell}{\ell^2-1}\right)_{\ell\in R}
		.
		$$
		Hence using Lemma~\ref{lem:limitsecond} and letting $R \to \Omega$ gives
		\[
		\lim_{D\to\infty}\frac{\#\Phi(D,\Omega)}{\#\Sigma(D)} \leq \lim_{D\to\infty}\frac{\#\Phi''(D,\Omega)}{\#\Sigma(D)}\leq \lim_{R\to\Omega}\lim_{D\to\infty}\frac{\#\Phi''(D,R)}{\#\Sigma(D)}=0. \qedhere
		\]
	\end{proof}


\begin{thebibliography}{xx}
		\bibitem{ABY}
		A. Akbary, J. Bleaney, S. Yazdani, On symmetries of elliptic nets and valuations of net polynomials. \emph{J. Number Theory} {\bf 158} (2016), 185--216. 
		
		\bibitem{AKY}
		A. Akbary, M. Kumar, S. Yazdani, Soroosh, The signs in elliptic nets. \textit{New York J. Math.} \textbf{23} (2017), 1237–1264.
		
		
		\bibitem{Aya93}
		M. Ayad, Périodicité (mod $q$) des suites elliptiques et points $S$-entiers sur les courbes elliptiques. \emph{Ann. Inst. Fourier} \textbf{43} (1993), no. 3, 585--618. 
		
		\bibitem{BL19}
		T. Browning, D. Loughran, Sieving rational points on varieties. \textit{Trans. Amer. Math. Soc.} \textbf{371} (2019), no. 8, 5757--5785.
		
		\bibitem{Bro19} J. Br\"udern, K. Matom\"aki, R. Vaughan, T. Wooley, {Analytic Number Theory.} Oberwolfach Rep. 16 (2019), 3141--3205. doi: 10.4171/OWR/2019/50
		
		\bibitem{BZ18}
		B. M. Bekker, Y. G. Zarkhin, Division by 2 of rational points on elliptic curves. \textit{St. Petersburg Math. J.} \textbf{29} (2018), no. 4, 683--713.
		
		\bibitem{bn19} J. Berg, M. Nakahara,
		{Rational points on conic bundles over elliptic curves.}
		Math. Z. \textbf{300} (2022), no. 3, 2429--2449.
		
		\bibitem{bg19}
		V. Bosser, {\'E}. Gaudron, Logarithmes des points rationnels des vari\'et\'es ab\'eliennes. \textit{Canadian Journal of Mathematics}, \textbf{71}(2)  (2019), 247-298.
		
		\bibitem{CTS20}
		J.-L. Colliot-Th\'el\`ene, A. Skorobogatov, 
		\emph{The Brauer--Grothendieck group}, Ergebnisse Mathematik 3.F., Springer, 2021.
		
		
		\bibitem{EPSW}
		G. Everest, A. van der Poorten, I. Shparlinski, T. Ward, \emph{Recurrence sequences}. Mathematical Surveys and Monographs, {\bf104}. American Mathematical Society, Providence, RI, 2003.
		
		\bibitem{ERS07} 
		G. Everest, J. Reynolds, S. Stevens, Shaun. On the denominators of rational points on elliptic curves. \textit{Bull. Lond. Math. Soc.} \textbf{39} (2007), no. 5, 762--770.
		
		\bibitem{et49}
		P. Erd\H{o}s, P. Tur\'an. On a problem in the theory of uniform distribution.  \textit{Indag. Math.} \textbf{10} (1949),  38--41. 
		
		\bibitem{HM73}
		{D. Husemoller, J. Milnor}, 
		\emph{Symmetric bilinear forms}.
		Springer-Verlag, New York-Heidelberg,
		1973.

		
		\bibitem{kow08} E. Kowalski, \emph{The large sieve and its applications.} Cambridge Tracts in Mathematics, \textbf{175.} Cambridge University Press, Cambridge, 2008.
		
		\bibitem{Lou13}
		{D. Loughran}, 
		{The number of varieties in a family which contain a rational point}.
		\emph{J. Eur. Math. Soc.}, {\bf20}(10) (2018), 2539--2588.
		
		\bibitem{LM19}
		D. Loughran, L. Matthiesen, Frobenian multiplicative functions and rational points in fibrations.
		\emph{J. Eur. Math. Soc.}, to appear, arXiv:1904.12845.
		
		\bibitem{ls16} D. Loughran, A.  Smeets,  {Fibrations with few rational points.} \emph{Geom. Funct. Anal.} \textbf{26} (2016), no. 5, 1449--1482.
		
		\bibitem{LM}
		D. Loughran on Mathoverflow,
		\url{https://mathoverflow.net/questions/374689/sum-of-fibonacci-sequence-evaluated-at-a-dirichlet-character}
		
		\bibitem{AIM02}
		W. McCallum, W. Stein, J. Voight, {Rational and Integral Points
			on Higher Dimensional Varieties.} Lecture notes for 
		ARCC workshop held at AIM in Palo Alto, December 11-20, 2002. \url{https://aimath.org/WWN/qptsurface2/}
		
		\bibitem{MV07} 
		H. L. Montgomery, R. C. Vaughan, \emph{Multiplicative number theory. I. Classical theory.} 
		Cambridge Studies in Advanced Mathematics, {\bf 97}. Cambridge University Press, Cambridge, 2007.
		
		\bibitem{Poo01}
		B. Poonen, An explicit algebraic family of genus-one curves violating the Hasse principle. 
		\emph{J. Théor. Nombres Bordeaux} {\bf13} (2001), no. 1, 263--274.

		
		\bibitem{Ser90}
		{J.-P. Serre}, {Sp\'{e}cialisation des \'{e}l\'{e}ments de }$\Br_2(\Q(T_1,\ldots,T_n))$,
		\emph{C. R. Acad. Sci. Paris S\'{e}r. I Math.} {\bf311} (1990), no. 7, 397--402.
		
		\bibitem{Shi00}
		R. Shipsey, Elliptic Divisibility Sequences. \emph{Ph.D. Thesis}, (2000).
		
		\bibitem{Sil88}
		J. H. Silverman, Wieferich's criterion and the $abc$-conjecture. \emph{J. Number Theory} 30 (1988), no. 2, 226--237.
		
		\bibitem{Sil05}
		J. H. Silverman, $p$-adic properties of division polynomials and elliptic divisibility sequences. \emph{Math. Ann.} \textbf{332} (2005), no. 2, 443--471.
		
		\bibitem{Sil09}
		J. H. Silverman, \emph{The arithmetic of elliptic curves}. Second edition. \emph{Graduate Texts in Mathematics} \textbf{106}. Springer, Dordrecht, (2009).
		
		\bibitem{SS06}
		J. H. Silverman, N. Stephens, The sign of an elliptic divisibility sequence. \emph{J. Ramanujan Math. Soc.} {\bf 21} (2006), no. 1, 1--17.
		
		\bibitem{Sko01}
		A. Skorobogatov, \textit{Torsors and rational points}.
		Cambridge University press, 2001.
		
		\bibitem{Sta06}
		K. E. Stange, Integral points on elliptic curves and explicit valuations of division polynomials. \emph{Canad. J. Math.} {\bf 68} (2016), no. 5, 1120--1158. 
		
		\bibitem{Sun}
		Z.-W. Sun on Mathoverflow,
		\url{https://mathoverflow.net/questions/301624/does-each-prime-p3-have-a-quadratic-nonresidue-which-is-a-mersenne-number}
		
		\bibitem{Va97}
		R. C. Vaughan, \textit{The Hardy-Littlewood Method}. 
		Cambridge University press, 1997.
		
		\bibitem{Ver21} 
		M. Verzobio, 
		A recurrence relation for elliptic divisibility sequences. 
		\textit{Riv. Math. Univ. Parma} \textbf{3} (2022), no. 1, 223--242.
		
		
		\bibitem{V00}
		J. F. Voloch,
		Elliptic Wieferich primes.
		\emph{J. Number Theory}.
		\textbf{81} (2000), no. 2,
		205--209.
		
		
		\bibitem{War48}
		M. Ward, Memoir on elliptic divisibility sequences. \emph{Amer. J. Math.} \textbf{70} (1948), 31--74.
		
		
	\end{thebibliography}
\end{document}